\documentclass{article}

\usepackage[utf8]{inputenc}
\usepackage[english]{babel}
\usepackage{amsfonts}
\usepackage[11pt]{moresize}
\usepackage{pgfplots}
\usepackage{xfrac}

\newcommand{\Addresses}{{
		\bigskip
		\footnotesize
		
		Goethe Universit\"at Frankfurt am Main, Institut f\"ur Mathematik, Robert-Mayer Strasse 6-8
		\hfill \newline\texttt{}
		\indent 60325 Frankfurt am Main, Germany} 
	\par\nopagebreak
	\textit{E-mail address}: \texttt{andreibud95@protonmail.com}
}

\usepackage{amssymb}
\usepackage{mathtools}
\usepackage{tikz}
\usepackage{tikz-cd}
\usepackage{mathtools}
\usepackage{amsmath}
\usepackage{amsthm}
\usepackage{hyperref}
\usepackage{float}
\usetikzlibrary{calc, positioning,decorations.markings,arrows}

\DeclarePairedDelimiter\floor{\lfloor}{\rfloor}

\usepackage{graphics}
\usepackage{color}
\usepackage{tabularx,colortbl}
\usetikzlibrary{positioning}
\usetikzlibrary{decorations.markings,arrows}
\usetikzlibrary{calc}
\usepackage{amsthm}

\theoremstyle{plain}
\newtheorem{trm}{Theorem}[section]
\newtheorem{lm}[trm]{Lemma}
\newtheorem{prop}[trm]{Proposition}
\newtheorem{cor}[trm]{Corollary}

\theoremstyle{definition}
\newtheorem{defi}[trm]{Definition}

\def\mzero{\mm_{0,g/\mathrm{S}_{g-1}}}

\def\OO{\mathcal{O}}
\def\cC{\mathcal{C}}

\def\cM{\mathcal{M}}
\def\cR{\mathcal{R}}
\def\rr{\overline{\mathcal{R}}}

\def\cC{\mathcal{C}}

\def\Pic0{{\rm Pic}^0(X)}

\def\mm{\overline{\mathcal{M}}}

\begin{document}
\title{The class of the Prym-Brill-Noether divisor}

\author{Andrei Bud}
\date{}
\maketitle
\begin{abstract}
	For $r\geq 3$ and $g= \frac{r(r+1)}{2}$, we study the Prym-Brill-Noether variety $V^r(C,\eta)$ associated to Prym curves $[C,\eta]$. The locus $\mathcal{R}_g^r$ in $\mathcal{R}_g$ parametrizing Prym curves $(C, \eta)$ with nonempty $V^r(C,\eta)$ is a divisor. We compute some key coefficients of the class $[\overline{\mathcal{R}}_g^r]$ in $\mathrm{Pic}_\mathbb{Q}(\overline{\mathcal{R}}_g)$. Furthermore, we examine a strongly Brill-Noether divisor in $\overline{\mathcal{M}}_{g-1,2}$: we show its irreducibility and compute some of its coefficients in $\mathrm{Pic}_\mathbb{Q}(\overline{\mathcal{M}}_{g-1,2})$. As a consequence of our results, the moduli space $\mathcal{R}_{14,2}$ is of general type. 
\end{abstract}
\section{Introduction} 
 The study of Prym curves from an algebraic perspective was initiated by Mumford in his seminal paper \cite{MumfordPrym}. Alongside Beauville's work \cite{Beau77}, where he provides a modular interpretation of Prym curves, these contributions laid the foundation for the study of the moduli space of Prym curves. This is defined as
  $$\mathcal{R}_g \coloneqq \{[C,\eta] \ \mid \ [C] \in \cM_g, \ \eta\in \mathrm{Pic}^0(C) \ \textrm{such that} \ \eta^{\otimes 2} \cong \OO_C \},$$ 
 parametrizing pairs $(C, \eta)$ where $C$ is a smooth curve of genus $g$ and $\eta$ is a $2$-torsion line bundle of $C$. 
 
 One natural question about $\cR_g$ is computing its Kodaira dimension. This problem was the focus of several mathematicians, who described the geometry of $\mathcal{R}_g$ for almost all values of $g$. This space is rational if $2\leq g \leq 4$, see \cite{Dolgachev}, \cite{Catanese}, unirational, if $5\leq g \leq 7$, see \cite{MoriMukai}, \cite{DonagiA5}, \cite{VerraA4}, \cite{VerraA5}, \cite{Izadi}, \cite{FarVerNikulin} uniruled if $g = 8$, see \cite{FarVerNikulin} and of general type if $g\geq 13, g \neq 16$, see \cite{FarLud}, \cite{Bruns}, \cite{FarR13}. 
 
  Through the natural map $\mathcal{P}_g\colon\mathcal{R}_g \rightarrow \mathcal{A}_{g-1}$, one can relate the geometry of principally polarized Abelian varieties to the geometry of curves. For $2\leq g \leq 6$ the map $\mathcal{P}_g$ is surjective, and hence the characterization above is fundamental in understanding the birational geometry of the moduli of Prym varieties. 
 
 Similarly, we can consider the moduli space $\mathcal{R}_{g,2}$ parametrizing tuples $(C,x+y,\eta)$ where $C$ is a smooth curve of genus $g$, the points $x$ and $y$ of $C$ are distinct, and $\eta$ is a line bundle satisfying $\eta^{\otimes 2} \cong \OO_C(-x-y)$. This space comes equipped with a map $\mathcal{P}_{g,2} \colon \cR_{g,2} \rightarrow \mathcal{A}_g$. This correspondence between pointed curves and principally polarized Abelian varieties motivates the study of the birational geometry of $\cR_{g,2}$. We know that $\mathcal{R}_{g,2}$ is unirational for $3\leq g \leq 5$, uniruled for $g= 6$ and of general type if $g \geq 16$ or $ g = 13$, see \cite{Lelli-Chiesa-uniruled}, \cite{BudKodPrym} and \cite{FarR13}.

 When studying the birational geometry of $\mathcal{M}_g$, Brill-Noether Theory plays a fundamental role in showing that $\mathcal{M}_g$ is of general type when $g \geq 22$, see \cite{KodMg}, \cite{KodevenHarris1984}, \cite{EisenbudHarrisg>23} and \cite{FarPayneJensen}. When $g\geq 24$, we can consider numbers $r, d$ such that $\rho(g,r,d) \coloneqq g-(r+1)(g-d+r)  = -1$ and look at the locus of curves $[C] \in \cM_g$ for which the Brill-Noether variety $W^r_d(C)$ is nonempty. This locus is a divisor in $\mathcal{M}_g$ and the class of its closure in $\mm_g$ can be used to show that $\mathcal{M}_g$ is of general type when $g\geq 24$. For a Prym curve $[C,\eta] \in \mathcal{R}_g$ we can consider $\pi\colon \widetilde{C} \rightarrow C$ the associated double cover and look at the locus 
\[V^r(C,\eta) \coloneqq \left\{ L \in \mathrm{Pic}^{2g-2}(\widetilde{C}) \ | \ \mathrm{Nm}(L)\cong \omega_C, \ h^0(\widetilde{C}, L) \geq r+1, \ \mathrm{and} \ h^0(\widetilde{C}, L) \equiv r+1\ (\mathrm{mod} \ 2)  \right\}\]
where the norm map sends a line bundle $L \in \textrm{Pic}(\widetilde{C})$ to $\wedge^2\pi_*L \otimes \eta$. Equivalently, it sends a line bundle $\OO_{\widetilde{C}}(D)$ to $\OO_C(\pi_*D)$ for every divisor $D$ on $\widetilde{C}$. 

These Prym-Brill-Noether loci can be understood as Brill-Noether loci on $\widetilde{C}$ that take into account the involution $\iota\colon \widetilde{C} \rightarrow \widetilde{C}$ associated to the double cover $\pi\colon \widetilde{C}\rightarrow C$. These loci were introduced in \cite{Welters} to understand the singularities of Prym varieties, particularly by computing the expected dimension and describing the smooth locus of $V^r(C,\eta)$. Subsequently, it was shown that for a generic $(C,\eta)$, the locus $V^r(C,\eta)$ has the expected dimension, see \cite{Bertram}, \cite{SchwarzPrym} and is irreducible when $g >\frac{r(r+1)}{2} + 1$, see \cite{DebarreLefschetz}. Viewing $V^r(C,\eta)$ as a Lagrangian degeneracy locus, De Concini and Pragacz computed the virtual class of this locus in the Prym variety, see \cite{DeConciniPragacz}. 

In recent years, two perspectives on the study of Prym-Brill-Noether loci arose. On one hand, the tropical geometry approach was used to provide another proof for the dimension estimate of $V^r(C,\eta)$, along with many other properties, see \cite{tropicalPBN-4authors}, \cite{tropicalPBN-Len_Ulirsch}, \cite{JensenPaynesurvey} and the references therein. On the other hand, the moduli theory approach was used to understand the birational geometry of $\cR_g$ for small values of $g$, see \cite{FarVerNikulin}. Moreover, for $g = \frac{r(r+1)}{2} +1$, it was shown in \cite{BudPrymIrr} that the universal Prym-Brill-Noether locus 
\[ \mathcal{V}^r_g\coloneqq \left\{[C,\eta,L] \ | \ [C,\eta]\in \mathcal{R}_g \ \mathrm{and} \ L \in V^r(C,\eta)\right\}\]
has a unique irreducible component dominating $\cR_g$.  

For $r\geq 3$ and $ g = \frac{r(r+1)}{2}$ we consider the locus 
\[ \cR^r_g \coloneqq \left\{ [C,\eta] \in \cR_g \ | \ V^r(C,\eta) \neq \emptyset \right\} \]
It follows immediately from \cite[Theorem II]{FultonLazarsfConnected}, \cite[Example 1.4]{DebarreLefschetz} and \cite[Theorem 2.9]{Steffen} that in this case $\cR^r_g$ is a divisor in $\cR_g$. The main goal of this paper is to compute the class of the Prym-Brill-Noether divisor $\rr^r_g$ in $\mathrm{Pic}_\mathbb{Q}(\rr_g)$, where the closure $\rr_g$ is described in \cite{Casa} and \cite{FarLud}. Our main result is: 

\begin{trm} \label{maintrm} Let $r \geq 3 $ and $g = \frac{r(r+1)}{2}$. Then the class of $\rr^r_g$ in $\mathrm{Pic}_\mathbb{Q}(\rr_g)$ is equal to 
\[ [\rr^r_g] = c\cdot (a\lambda- b_0'\delta_0' - b_0''\delta_0''-b_0^{\mathrm{ram}}\delta_0^{\mathrm{ram}}- \sum_{i=1}^{g-1}b_i\delta_i - \sum_{i = 1}^{[\frac{g}{2}]}b_{i:g-i}\delta_{i:g-i}) \]
where $a = g+1$, $b_0' = \frac{g}{6}$, $b_0^{\mathrm{ram}} = \frac{g}{4}$, $b_0'' = \frac{g^2-g+2}{8}$ and $b_i = \frac{(g-i)(g+i-1)}{2}$.
\end{trm}
The constants $c$ and $b_{i:g-i}$ were not determined.

Theorem \ref{maintrm} provides another proof of the fact that $\cR_{15}$ is of general type, proven in \cite{Bruns}. Moreover, by pulling back this divisor to $\cR_{14,2}$ we are able to prove that

\begin{trm} \label{Kodaira} The moduli space $\mathcal{R}_{14,2}$ is of general type. 
\end{trm}

Using the numerology of Theorem \ref{maintrm}, we can intersect the Prym-Brill-Noether divisor with a pencil of Prym curves on a Nikulin surface. Because the intersection number is negative, we obtain the following result about the Nikulin locus (i.e. the locus of Prym curves lying on Nikulin surfaces) in $\mathcal{R}_g$:

\begin{cor} \label{nikulin} Let $r\geq 3$ and $g = \frac{r(r+1)}{2}$. Then the Nikulin locus is contained in the Prym-Brill-Noether divisor. 
\end{cor} 

This is another proof that Prym-Brill-Noether varieties do not have the expected dimension on Prym curves in the Nikulin locus. A more general version of this result, proved using the geometry of Nikulin surfaces in an essential way, appears in \cite{Lelli-Chiesa-lowgenus}. 

In order to prove Theorem \ref{maintrm}, we will consider the intersection of $\rr^r_g$ with the boundary divisor $\Delta''_0$. To understand this intersection we will work with Prym limit linear series for curves that are not of compact type. The theory developed by Osserman in \cite{Ossermandim} and \cite{Ossermancompacttype} is well-suited to tackle this problem. The norm condition on the limit linear series will substantially simplify the situation. To compute the class of $\rr^r_g$, we will have to compute the class of a strongly Brill-Noether divisor in $\mm_{g-1,2}$. 

For $x,y$ two points on a curve $C$, we consider the sequence $D_\bullet(x,y)$ of effective divisors: 
\[ 0 \leq x+y \leq 2\cdot(x+y) \leq \cdots \leq n\cdot (x+y) \leq \cdots\]
and the multivanishing sequence $\textbf{a}$:
\[ a_0 = 0 \leq a_1 = 2 \leq a_2 = 4 \leq \cdots \leq a_r = 2r  \] 
For $r \geq 3$ and $ g = \frac{r(r+1)}{2} - 1$, we consider the locus in $\cM_{g,2}$ of pointed curves $[C, x, y]$ satisfying that $C$ admits a $g^r_{g+r}$ with multivanishing sequence $\textbf{a}$ along  $D_\bullet(x,y)$. That is: 
\begin{multline*} \cM^r_{g, g+r}\text{\large (}D_\bullet, \textbf{a}\text{\large )} \coloneqq \left\{ \right. [C,x,y] \in \cM_{g,2} \ | \ \exists \ L\in W^r_{g+r}(C) \ \textrm{satisfying} \ \\  h^0\text{\Large (}C,L(-i(x+y))\text{\Large )} \geq r+1-i \ \forall \ 0\leq i \leq r \left.\right\} 
\end{multline*}

This locus has a divisorial component and we can show 

\begin{trm} \normalfont \label{stronglyBNdiv} In the notation above, the strongly Brill-Noether divisor is irreducible and its class satisfies: 
\[[\mm^r_{g, g+r}\text{\large (}D_\bullet, \textbf{a}\text{\large )}] = c\cdot(a_1\psi_1 + a_2\psi_2 + a\lambda - b_0\delta_0 - \sum_{i=0}^{g-1} b_{i, \left\{1,2\right\}}\delta_{i, \left\{1,2\right\}} - \sum_{i=1}^{g-1}b_{i,1}\delta_{i,1}) \]
where $a_1 = a_2 = \frac{g^2+g+2}{8}$, $a = g+2$, $b_0 = \frac{g+1}{6}$, $b_{i,\left\{1,2\right\}} = \frac{(g-i)(g+i+1)}{2}$ and $c = \frac{(g+1)!}{g-1}\cdot 2^{g-1}\prod_{i=1}^{r}\frac{i!}{(2i)!}$. 
\end{trm}

The coefficients $b_{i,1}$ for $1\leq i \leq g-1$ were not determined.



In order to prove Theorem \ref{maintrm}, several basic Brill-Noether properties will be required. We provide these results in Section \ref{basicBN}. Next, we consider in Section \ref{testintersection}, the intersection of the divisor $\rr^r_g$ with several test curves. The interplay between the norm condition, the Brill-Noether number and limit linear series is first investigated in this section. In Section \ref{pullbacks}, we consider different pullbacks of the divisor $\rr^r_g$. These pullbacks consist of a unique non-boundary divisor and hence, will provide new relations between the coefficients of the class $[\rr^r_g]$ in $\mathrm{Pic}_\mathbb{Q}(\rr_g)$. The results in Section \ref{testintersection} and Section \ref{pullbacks} conclude Theorem \ref{maintrm} and Theorem \ref{Kodaira}. Finally in Section \ref{stronglysection} we deal with strongly Brill-Noether divisors in $\mm_{g,2}$. For a generic curve $[C]\in \cM_g$, the fibre of the strongly Brill-Noether divisor appearing in Theorem \ref{stronglyBNdiv} is one dimensional above $[C]$. We consider the locus of tuples $(x,y, L)$ satisfying $[C,x,y] \in \mm^r_{g, g+r}\text{\large (}D_{\bullet}, \textbf{a}\text{\large )}$, $L\in  \mathrm{Pic}^{g+r}(C)$ and respecting the condition in the definition of $\mm^r_{g, g+r}\text{\large (}D_{\bullet}, \textbf{a}\text{\large )}$. This is a one dimensional locus in the product space $C\times C \times \mathrm{Pic}^{g+r}(C)$ and can be realized as a flag degeneracy locus. We use the Fulton-Pragacz determinantal formula to compute the intersection of this locus with the divisors $\Delta \times \mathrm{Pic}^{g+r}(C)$ and $C\times \left\{p\right\} \times \mathrm{Pic}^{g+r}(C)$. This gives us the irreducibility , together with a relation between its coefficients in $\mathrm{Pic}_\mathbb{Q}(\mm_{g,2})$ that will allow us to compute the coefficients of $\psi_1$ and $\psi_2$. The intersection of $\mm^r_{g, g+r}\text{\large (}D_{\bullet}, \textbf{a}\text{\large )}$ with the boundary divisor $\Delta_{0,\left\{1,2\right\}}$ is easy to understand and will be used to conclude Theorem \ref{stronglyBNdiv}.

\textbf{Acknowledgements:} This paper was part of my PhD project at Humboldt Universit\"at zu Berlin. I am very grateful to Gavril Farkas for suggesting this topic, and to Andr\'es Rojas for the helpful comments on this paper. I am also thankful to the anonymous referee for the diligent reading of this paper and for the many suggested improvements. The author acknowledges support by Deutsche Forschungsgemeinschaft
(DFG, German Research Foundation) through the Collaborative Research
Centre TRR 326 \textit{Geometry and Arithmetic of Uniformized Structures}, project number 444845124.

 \section{Some basic Brill-Noether properties} \label{basicBN}
 
 To understand the intersection of $\overline{\mathcal{R}}^r_g$ with different divisors, we will require several well-know Brill-Noether properties, that we will recall in this section. We start by reviewing some basic definitions about linear series. 
  
 In their seminal work on Brill--Noether Theory (see \cite{limitlinearbasic}), Eisenbud and Harris wanted to study line bundles of a given degree possessing numerous global sections. To accomplish this, they used the concept of a linear series:
 
 \begin{defi} Let $C$ be a smooth curve of genus $g$. A linear series $g^r_d$ on $X$ is a pair $l = (L, V)$ where $L \in \textrm{Pic}^d(C)$ is a degree $d$ line bundle and $V\subseteq H^0(C,L)$ is an $(r+1)$-dimensional subspace of the space of global sections on $L$. The variety parametrizing all $g^r_d$'s on a curve $C$ is denoted $G^r_d(C)$ 
 \end{defi}
 
 Having some points $x_1, \ldots, x_n$ on $C$, it is natural to look at their vanishing orders with respect to linear series. 
 
 \begin{defi}
Let $\textbf{a}$ be a sequence $0\leq a_0< \cdots < a_r \leq d$. We say that a $g^r_d$, denoted $(V,L)$, has vanishing orders given by $\textbf{a}$ at a point $x$ if there exists a basis of $V$ such that the vanishing orders at $x$ of that basis are the $a_i$'s for $0\leq i \leq r$. 
 \end{defi}

 If $\textbf{a}^1, \ldots, \textbf{a}^n$ are ramification profiles, we define the Brill-Noether number to be 
 \[ \rho(g,r,d,\textbf{a}^1,\ldots,\textbf{a}^n) = g-(r+1)(g-d+r)- \sum_{j=1}^{n}\sum_{i=0}^{r}a^j_i + n\cdot \frac{r(r+1)}{2}\]
 
The main result of \cite{limitlinearbasic} is a partial compactification of the space of linear series to curves that are of compact type. For a definition of limit linear series, we refer the reader to \cite{limitlinearbasic}.
 
 We have the following results, compiled from \cite[Theorem 1.1]{EisenbudHarrisg>23} and \cite[Proposition 1.4.1]{FarkasThesis}.
 \begin{lm}\label{bnLemma} On a genus $g$ curve, we consider limit linear $g^r_d$'s having vanishing profiles $\textbf{a}^1, \ldots, \textbf{a}^n$ at $n$ marked points. Then we have the following: \\
 I) If $g=0$ and $[R,x_1,\ldots,x_n] \in \mm_{0,n}$ admits such a limit $g^r_d$, then the Brill-Noether number is positive, i.e: 
 \[ \rho(0,r,d,\textbf{a}^1,\ldots,\textbf{a}^n) \geq 0\]
 II) If $g=1, n=1$ and $[E,x] \in \mm_{1,1}$ admits such a $g^r_d$, then
  \[ \rho(1,r,d,\textbf{a}) \geq 0\]
 III) If $g=1, n=2$ and $[E,x,y] \in \mm_{1,2}$ admits such a $g_d^r$, then 
 \[ \rho(1,r,d,\textbf{a}^1, \textbf{a}^2) \geq -r\]
 IV) Let $\mathcal{W} \subseteq \cM_{2,1}$ be the Weierstrass divisor. If $g=2, n= 1$ and $[C,x] \in \cM_{2,1}\setminus \mathcal{W}$ admits such a $g^r_d$, we have: 
  \[ \rho(2,r,d,\textbf{a}) \geq 0\]
 V) If $g=2, n =1 $ and $[C,x] \in \partial\mm_{2,1}$ is a generic point of a boundary component admitting such a limit $g^r_d$, then: 
 \[ \rho(2,r,d,\textbf{a}) \geq 0 \] \color{black}
 VI) If $[C,x_1,\ldots, x_n] \in \cM_{g,n}$ is generic and admits such a $g^r_d$, then 
  \[\rho(g,r,d,\textbf{a}^1,\ldots,\textbf{a}^n) \geq 0\]
 \end{lm}

Our next goal is to understand the Brill-Noether theory of Prym curves. For this, we provide a pointed version of the main result in \cite{SchwarzPrym}. As in \cite{BudKodPrym}, we denote by $\mathcal{C}^n\mathcal{R}_g\coloneqq \mathcal{R}_g\times_{\mathcal{M}_g}\cM_{g,n}$ the moduli space parametrizing tuples $[C,x_1,\ldots,x_n, \eta]$ where $[C, \eta] \in \mathcal{R}_g$ and $x_1, \ldots, x_n \in C$. We have:
\begin{prop} \label{pointedSchwarz}
	For a generic pointed Prym curve $[C,x,\eta] \in \mathcal{C}^1\mathcal{R}_g$, let $\widetilde{C}\rightarrow C$ be the associated double cover and let $\widetilde{x}_1,\widetilde{x}_2 \in \widetilde{C}$ the two points in the preimage of $x$.   
	
	We consider some integers $r, d$ and some vanishing profiles $\textbf{a}^1, \textbf{a}^2$ such that the condition 
	\[\rho(g,r,d,\textbf{a}^1, \textbf{a}^2) <-r\] 
	is satisfied. Then $\widetilde{C}$ does not admit a $g^r_d$ with ramification profiles $\textbf{a}^1$ and $\textbf{a}^2$ at $\widetilde{x}_1,\widetilde{x}_2$.    
\end{prop}
\begin{proof} We consider the map $\chi^1_g\colon \mathcal{C}^1\mathcal{R}_g \rightarrow \cM_{2g-1,2/S_2}$ sending $[C,x,\eta]$ to $[\widetilde{C},\widetilde{x}_1+\widetilde{x}_2]$. This map can be extended to a map 
	\[ \chi^1_g\colon \overline{\mathcal{C}^1\cR}_g\rightarrow \mm_{2g-1,2/S_2} \]
where the compactification of $\mathcal{C}^1\mathcal{R}_g$ is as in \cite[Section 6]{BudKodPrym}. We consider $[X\cup_{y\sim p}E, x, \OO_X, \eta_E]$ a generic point in the boundary divisor $\Delta_1$ of $\overline{\mathcal{C}^1\cR}_g$. The image of this point through $\chi_g$ is $[X_1\cup_{y_1\sim p_1} \widetilde{E} \cup_{p_2\sim y_2} X_2, x_1, x_2]$ where $[X_1, x_1, y_1]$ and $[X_2, x_2, y_2]$ are two copies of the generic curve $[X,x,y] \in \cM_{g-1,2}$ and $[\widetilde{E},p_1,p_2]$ is the associated double cover of $[E,p,\eta_E]$, that is $p_1$ and $p_2$ are the points in the preimage of $p$ for the associated double cover. 

If we assume the proposition to be false, we get that $[X_1\cup_{y_1\sim p_1} \widetilde{E} \cup_{p_2\sim y_2} X_2, x_1, x_2]$ admits a limit $g^r_d$ having ramification profiles $\textbf{a}^1$ and $\textbf{a}^2$ at $x_1$ and $x_2$. We denote by $l_1, l_2$ and $l_{\widetilde{E}}$ the aspects of this limit linear series. 

Using the additivity of the Brill-Noether numbers, see \cite[Proposition 4.6]{limitlinearbasic}, together with III and VI of Lemma \ref{bnLemma}, we obtain the contradiction 
\[ -r>\rho(g,r,d,\textbf{a}^1,\textbf{a}^2) \geq \rho(l_1,x_1,y_1) + \rho(l_2,x_2,y_2) + \rho(l_{\widetilde{E}},p_1,p_2) \geq   0+0+(-r) = -r.\]
\end{proof}

To understand how Prym-Brill-Noether loci degenerate to the boundary component $\Delta_0''$, we will require the study of multivanishing orders (with respect to a chain of divisors).

\begin{defi}
Let $l = (L, V)$ be a $g^r_d$ on $C$ and let $\textbf{D}$ be a chain of effective divisors on $C$:
\[ 0 = D_0 < D_1 <\cdots < D_k \]
satisfying $\deg(D_k) > d$. We say that a section $s \in V$ has \emph{multivanishing order~$\deg(D_i)$} with respect to $\textbf{D}$ if 
\[ s \in V\cap H^0(C,L-D_i) \ \textrm{and} \ s \not\in V\cap H^0(C,L-D_{i+1}). \]
 
As before, there are exactly $r+1$ multivanishing orders, giving a
\emph{multivanishing sequence}
\[a^{\ell} (\textbf{D}): 0\leq a_0^{\ell}(\textbf{D}) < a_1^{\ell}(\textbf{D})\cdots < a_r^{\ell}(\textbf{D}) \leq d\]
with respect to~$\textbf{D}$.
\end{defi}

Notice that in this situation, there can exist multiple independent sections having the same multivanishing order $\deg(D_i)$. In fact, there can exist at most $\deg(D_{i+1}) - \deg(D_i)$ such sections. 

Let $\textbf{a}$ be a sequence $0\leq a_0 \leq a_1\leq\cdots \leq a_r \leq d$ and let $r_i$ be the number of times that $i$ appear in this sequence. In this case, the Brill-Noether number is defined as 

\[ \rho(g,r,d,\textbf{a}) = g-(r+1)(g-d+r)- \sum_{j=1}^{n}\sum_{i=0}^{r}a_i + \cdot \frac{r(r+1)}{2} - \sum_{i=0}^d \binom{r_i}{2}.\]

This number represent the expected dimension for the variety parametrizing $g^r_d$'s with multivanishing orders $\textbf{a}$ with respect to a chain of divisors $\textbf{D}$. When this number is negative, a generic pointed curve does not admit such $g^r_d$'s, see \cite{Ossermancompacttype}. If all the Brill-Noether varieties of $g^r_d$'s respecting a multivanishing condition are of expected dimension for a pointed curve $[C, x_1,\ldots, x_n]$ we call the pointed curve \textit{strongly Brill-Noether general}.
\section{Intersection with test curves} \label{testintersection}
A standard way of obtaining relations between the coefficients of a divisor is to intersect it with different test curves. One way to obtain test curves on the moduli space $\rr_g$ is to pullback known test curves in $\mm_g$. This approach was already employed in \cite{FarLud}, \cite{Carlos}, \cite{Bud-adm} and \cite{Bud-newdiv}. We start by defining the test curves we will use in this section. 

Let $[X,p]$ be a generic genus $g-1$ pointed curve. The test curve $A$ in $\mm_g$ is obtained by glueing at the point $p$ an elliptic pencil along a base point. Pulling-back the test curve $A$ to $\rr_g$ we obtain three test curves $A_1$, $A_{g-1}$ and $A_{1:g-1}$ contained in the boundary divisors $\Delta_1$, $\Delta_{g-1}$ and $\Delta_{1:g-1}$ respectively.

Let $g = \frac{r(r+1)}{2}$ and $\cR_g^r$ the locus parametrizing curves $[C,\eta]$ for which $V^r(C,\eta)$ is non-empty. We denote by $\rr^r_g$ the closure of this locus in $\rr_g$. We consider the map $\chi_g\colon \rr_g \rightarrow \mm_{2g-1}$ sending a Prym curve $[C,\eta]$ to the associated double cover $\widetilde{C}$ of $C$. Using this map, we prove:
\begin{prop} \label{testA_g-1}
	We have the intersection number  $A_{g-1} \cdot \rr_g^r = 0$.
\end{prop} 
\begin{proof}
	By definition, we have that $\rr^r_g \subseteq \chi_g^{-1}(\mm_{2g-1,2g-2}^r)$. 
	To conclude our proposition, it is enough to show that the curves in $\chi_g(A_{g-1})$ do not admit any limit $g^r_{2g-2}$. 
		
	The fact that  $\chi_g(A_{g-1})$ and $\mm_{2g-1,2g-2}^r$ do not intersect follows from Proposition \ref{pointedSchwarz} and part II of Lemma \ref{bnLemma}. The conclusion follows from the additivity of Brill-Noether numbers. 
\end{proof}

We also have that the test curve $A_1$ and the divisor $\rr^r_g$ do not intersect. However, the proof is more involved due to the following fact: If we look at the element in the intersection of $A_1$ and $\Delta_0^{\mathrm{ram}}$, the associated double cover is of pseudo-compact type but not of compact type. However, we can describe this double cover, and use the theory of limit linear series for curves not of compact type to conclude that the curve does not admit a limit $g^r_{2g-2}$. We refer the reader to \cite{Ossermandim} and \cite{Ossermancompacttype} for more details on limit linear series for curves not of compact type.
\begin{prop} \label{testA_1}
	We have the intersection number $A_1\cdot \rr_g^r = 0$.
\end{prop}
\begin{proof}
		We assume a curve of compact type in $\chi_g(A_1)$ admits such a limit $g^r_{2g-2}$. Using parts III and VI of Lemma \ref{bnLemma}, together with the additivity of Brill-Noether numbers, we get the contradiction 
	\[ \rho(2g-1,r,2g-2) = -r-2 \geq -r \]
	
The only curve in $A_1$ not associated to a double cover of compact type is the one in $A_1\cap \Delta_0^{\mathrm{ram}}$. Let $[X_1, p_1]$ and $[X_2, p_2]$ two copies of the generic curve $[X,p]$ used in the test curve and let $[R_1, x_1, y_1,z_1]$ and $[R_2, x_2,y_2,z_2]$ two copies of the unique element of $\mathcal{M}_{0,3}$. Then the associated double cover for the curve in $A_1\cap \Delta_0^{\mathrm{ram}}$ is obtained from the curves defined above by glueing together $y_1$ with $y_2$, $z_1$ with $z_2$ and $p_i$ with $x_i$ for $i=1,2$. We denote this curve by $\widetilde{C}$ and the target of the double cover by $C$. The dual graph $\Gamma(\widetilde{C})$ of this curve is 
\begin{figure}[H] \centering
\begin{tikzpicture}[auto, node distance=3cm, every loop/.style={},
thick,main node/.style={circle,draw,font=\sffamily\Large\bfseries}]
	
	\node[main node] (1) {$g-1$};
	\node[main node] (2) [right of=1] {$0$};
	\node[main node] (3) [right of=2] {$0$};
	\node[main node] (4) [right of=3] {$g-1$};
	
	\path[every node/.style={font=\sffamily\small}]
	(1) edge node [] {} (2)
	(2) edge[bend right] node [left] {} (3)
	(3) edge node [right] {} (4)
	(2) edge[bend left] node [right] {} (3);
	\node[above right = 0.5mm of {(4.2,0.5)}] {$y$};
	\node[above right = 0.5mm of {(4.2,-1)}] {$z$};
\end{tikzpicture} \caption{The dual graph of $\widetilde{C}$, decorated with genera of the components}
\end{figure}
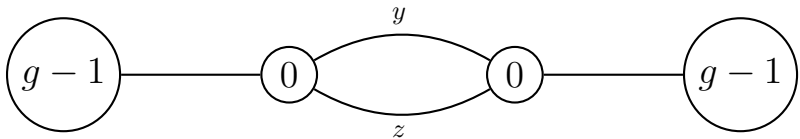 

As remarked in \cite[Theorem 3.3]{Ossermandim}, all components of the curve $\widetilde{C}$ are strongly Brill-Noether general. Next we want to understand how does a linear series $g^r_{2g-2}$ satisfying the norm condition specialize to $\widetilde{C}$. To understand the possible limit linear series above this curve, we look at \cite[Section 3]{Ossermancompacttype}. We assume there exists a Prym limit linear series on $\widetilde{C}$ and we study what multivanishing conditions such a limit linear series satisfies.


We consider a smoothing family of $\pi\colon \widetilde{C} \rightarrow C$ in $\rr_g$
	\[
\begin{tikzcd}
	\widetilde{\mathcal{C}}  \arrow{r}{\pi} \arrow[swap]{dr}{} & \mathcal{C} \arrow{d}{} \\
	&  \Delta
\end{tikzcd}
\]
where $\Delta$ is the unit disk. Let $\Delta^*$ be the disk without the origin and assume that $\widetilde{\mathcal{C}}^*\coloneqq \widetilde{\mathcal{C}} \times_{\Delta}\Delta^*$ admits a line bundle $\mathcal{L}^*$ such that 
\[ \mathrm{Nm}_{\pi}\mathcal{L}^* \cong \omega_{\mathcal{C}^*/\Delta^*} \]

Let $\mathcal{C}'\rightarrow \mathcal{C}$ be the crepant resolution that smooths the singularity at the non-separating node of the central fibre $C_0 = C$. We consider $\widetilde{\mathcal{C}}'\coloneqq \widetilde{\mathcal{C}}\times_{\mathcal{C}} \mathcal{C}'$ and observe that the pullback of $\mathcal{L}^*$ to this space can be extended over the central fibre. Let $\mathcal{L}$ be a line bundle on $\widetilde{\mathcal{C}}'$ so obtained. Because $\pi' \colon \widetilde{\mathcal{C}}' \rightarrow \mathcal{C}'$ is an \'etale double cover, it follows from \cite[6.5.2]{EGA2} that the norm is well-defined and $\mathrm{Nm}_{\pi'}(\mathcal{L})$ is a line bundle on $\mathcal{C}'$ that extends $\omega_{\mathcal{C}'/\Delta^*}$. Hence we have 
\[ \mathrm{Nm}_{\pi'}(\mathcal{L}) \cong \omega_{\mathcal{C}'/\Delta}(\sum D_i)\]
where $D_i$ are irreducible components of the central fibre. 

We look at the chains of rational curves added when smoothing the nodes $y$ and $z$ of the central fibre $\widetilde{C}$. Up to tensoring with irreducible components of the central fibre, we may assume that $\mathcal{L}$: 
\begin{itemize}
	\item restricts to the trivial line bundle on all but at most one rational component in the two chains and 
	\item if it restricts non-trivially to a rational component, then it has degree $1$ on it. 
\end{itemize}

We know from \cite[Proposition 6.5.8]{EGA2} that the norm map is well-behaved with respect to restricting to the central fibre. When looking at the degrees of $\omega_{\mathcal{C}'/\Delta}(\sum D_i)$ on the rational components in the chain, they add up to an even number. From here, it follows that $\mathcal{L}$ must be trivial on the chains added at the nodes $y$ and $z$. 

By construction, the chains have the same number of irreducible components. Hence, the multivanishing orders of the Prym limit linear series (on $R_1$ and $R_2$) are with respect to the following two sequences of divisors:
\[ 0 \leq y_1+ z_1 \leq \cdots \leq g\cdot (y_1+ z_1) \]
and
\[ 0 \leq y_2+ z_2 \leq \cdots \leq g\cdot (y_2+ z_2). \]
 This is a consequence of \cite[Theorem 5.9]{Ossermandim}. 

We have two possibilities for the concentrated multidegrees. It is enough to describe the possible multidegrees concentrated at $X_1$. The ones for the other components are obtained from those by twisting. 

The possible multidegrees concentrated at $X_1$ are: 
\begin{figure}[H] \centering
	\begin{tikzpicture}[auto, node distance=3cm, every loop/.style={},
		thick,main node/.style={circle,draw,font=\sffamily\Large\bfseries}]
		
		\node[main node] (1) {$2g-3$};
		\node[main node] (2) [right of=1] {$0$};
		\node[main node] (3) [right of=2] {$1$};
		\node[main node] (4) [right of=3] {$0$};
		
		\path[every node/.style={font=\sffamily\small}]
		(1) edge node [] {} (2)
		(2) edge[bend right] node [left] {} (3)
		(3) edge node [right] {} (4)
		(2) edge[bend left] node [right] {} (3);
	\end{tikzpicture} \caption{Concentrated multidegree: first possibility}
\end{figure}
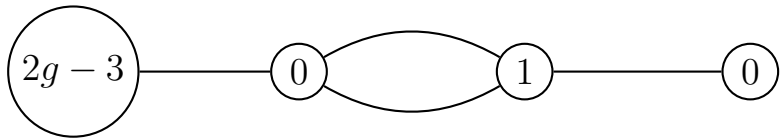 
and 
\begin{figure}[H] \centering
	\begin{tikzpicture}[auto, node distance=3cm, every loop/.style={},
		thick,main node/.style={circle,draw,font=\sffamily\Large\bfseries}]
		
		\node[main node] (1) {$2g-2$};
		\node[main node] (2) [right of=1] {$0$};
		\node[main node] (3) [right of=2] {$0$};
		\node[main node] (4) [right of=3] {$0$};
		
		\path[every node/.style={font=\sffamily\small}]
		(1) edge node [] {} (2)
		(2) edge[bend right] node [left] {} (3)
		(3) edge node [right] {} (4)
		(2) edge[bend left] node [right] {} (3);
	\end{tikzpicture} \caption{Concentrated multidegree: second possibility}
\end{figure}
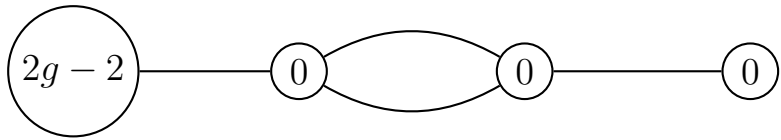 

Assume we are in the first case. Then, on the generic curve $[X_1, p_1]$ we have a $g^r_{2g-3}$ with ramification orders at $p_1$ denoted $0\leq a_0^1 < a_1^1 <\cdots < a_r^1 \leq 2g-3$. 

The genericity of $[X_1,p_1]$ implies 
\[ \rho(g-1,r,2g-3) - \sum_{i=0}^r a_i^1 +\frac{r(r+1)}{2} \geq 0 \]
That is: 
\[ g-1 + (r+1)(g-2-r) + \frac{r(r+1)}{2} \geq \sum_{i=0}^r a^1_i\]
Let $b_0^1,\ldots, b_r^1$ be the ramification orders of the limit linear series at $x_1 \in R_1$. We know 
\[ \sum_{i=0}^r(b_{r-i}^1 + a_i^1) \geq (r+1)(2g-3) \]
From this and the previous inequality we get: 
\[ \sum_{i=0}^rb_i^1 \geq (r+1)(g+r-1) - \frac{r(r+1)}{2} + 1-g \]

If we denote $b_0^2,\ldots, b_r^2$ to be the ramification orders of the limit linear series at $x_2 \in R_2$, we obtain analogously: 
\[ \sum_{i=0}^rb_i^2 \geq (r+1)(g+r-1) - \frac{r(r+1)}{2} + 1-g \] 

We denote by $c^1_0, c^1_1,\ldots, c^1_r$ the multivanishing orders associated to $R_1$ for the sequence of divisors
\[ 0 \leq y_1+ z_1 \leq \cdots \leq g\cdot (y_1+ z_1) \]
We consider $c^2_0,\ldots, c_r^2$ similarly for the rational component $R_2$. 

Because $[R_1, x_1, y_1, z_1]$ is strongly Brill-Noether general, see \cite[Definition 3.2 and Theorem 3.3]{Ossermandim}, it follows that: 
\[ \rho(0,r,2g-3) - \sum_{i=0}^rb_i^1 - \sum_{i=0}^r c_i^1 + r(r+1) \geq 0 \]
Hence 
\[ (r+1)(2g-3) - \sum_{i=0}^rb_i^1 \geq \sum_{i=0}^r c_i^1 \]
Similarly
\[ (r+1)(2g-3) - \sum_{i=0}^rb_2^1 \geq \sum_{i=0}^r c_2^1 \]
Adding the two formulas and using the compatibility condition, see \cite[Definition 2.16]{Ossermandim} we get the contradiction 
\[ (r+1)(2g-4)-2 \geq \sum_{i=0}^{r}(c^1_i+c^2_i) \geq (r+1)(2g-4) \]

The second possibility for the multidegrees is treated analogously. In conclusion, this curve does not admit a Prym limit $g^r_{2g-2}$.
\end{proof}
\section{Pullbacks of the Prym-Brill-Noether divisor} \label{pullbacks}

Another standard way of obtaining relations between the coefficients of a divisor is to understand its pullbacks through different maps. We will separate this section into two, depending on whether the norm condition is necessary in understanding the pullback, or the Brill-Noether number suffices.  
\subsection{Pullbacks and Brill-Noether theory}
Let $\mm_{0,g/\mathrm{S}_{g-1}}$ be the moduli space parametrizing stable $g$-pointed genus 0 curves $[R, p_1+\cdots+p_{g-1}, p_g]$ where the markings $p_1,\ldots, p_{g-1}$ are unordered. On this moduli space, we have the boundary divisors $\epsilon_2,\ldots, \epsilon_{g-2}$, where a generic element of $\epsilon_i$ has two irreducible components and the point $p_g$ is on a component with exactly $i-1$ other markings. Moreover, we consider an elliptic pointed Prym curve $[E,x,\eta_E] \in \cC^1\cR_1$ and take the map: 
\[ i \colon \mzero \rightarrow \rr_g   \]
glueing a copy of $[E,x,\OO_E]$ to each of the points $p_1,\ldots, p_{g-1}$ and a copy of $[E,x,\eta_E]$ to $p_g$. First, we describe the pullback of this map at the level of divisors. 
\begin{prop} \label{Picardmap} Let $ i \colon \mzero \rightarrow \rr_g $ be the map above. Then we have:
	\begin{itemize}
		\item $ i^*\lambda = i^*\delta_0'=i^*\delta_0'' =i^*\delta^{\mathrm{ram}}_0 = 0$
		\item $i^*\delta_{i:g-i} = 0 \ \textrm{for every} \ 1\leq i \leq \floor{\frac{g}{2}}$
		\item  $i^*\delta_i = \epsilon_i \ \textrm{for} \ 2\leq i \leq g-2$
		\item $i^*\delta_{g-1} = -\sum^{g-2}_{i=2}\frac{(i-1)(g-i)}{g-2}\epsilon_i \ \textrm{and}$
		\item $i^*\delta_1 = - \sum_{i=2}^{g-2}\frac{(g-i-1)(g-i)}{(g-2)(g-1)}\epsilon_i$
	\end{itemize}
\end{prop}
\begin{proof}
	All but the last two formulas follow by simple geometric observations and by looking at the composition map 
	\[ \mzero \xrightarrow{i} \rr_g \rightarrow \mm_g  \]
	whose pullback at the level of Picard groups was computed in \cite{EisenbudHarrisg>23}. 
	
	For computing $i^*\delta_1$ and $ i^*\delta_{g-1}$ we look at the diagram: 
	\[
\begin{tikzcd}
	\mzero \arrow{r}{j}  \arrow[swap]{d}{p} & \mm_{g-1,1} \arrow{d}{} \arrow{r}{\pi} & \rr_{g} \\
	\overline{\mathcal{M}}_{0,g-1/S_{g-1}}\arrow{r}{}& \mm_{g-1}
\end{tikzcd}
\] 
	
	We know from \cite[Proposition 6.1]{BudKodPrym} that $\pi^*\delta_1 = -\psi$ and $\pi^*\delta_{g-1} = \delta_{g-2}$. Using this and \cite[Section 3]{EisenbudHarrisg>23} we get $i^*\delta_1 = -\psi_g$ and $i^*\delta_{g-1} = -\sum_{i=1}^{g-1}\psi_i$. Furthermore
	\[ -\sum_{i=1}^{g-1}\psi_i = p^*(-\sum_{i=1}^{g-1}\psi_i) -\epsilon_2 = p^*(-\sum_{i=2}^{\floor{\frac{g-1}{2}}}\frac{i(g-1-i)}{g-2} \epsilon_i ) -\epsilon_2 \]
But $p^*\epsilon_i = \epsilon_{i+1} + \epsilon_{g-i}$ with the exception $i=\frac{g}{2}$ when $g$ is even, in which case $p^*\epsilon_{\frac{g}{2}} = \epsilon_{\frac{g}{2}+1}$. Consequently we have 
\[ i^*\delta_{g-1} = -\sum_{i=2}^{g-2}\frac{(i-1)(g-i)}{g-2}\epsilon_i \] 
Because $\sum_{i=1}^g \psi_i = \sum_{i=2}^{g-2}\epsilon_i$ we get 
\[ i^*\delta_1 = -\psi_g =  - \sum_{i=2}^{g-2} \frac{(g-i-1)(g-i)}{(g-2)(g-1)}\epsilon_i \]
\end{proof}

We have the following: 
\begin{prop} \label{zeropullback}
	Let $i \colon \mzero \rightarrow \rr_g$ be as above. Then we have $i^*[\rr^r_g] = 0$.
\end{prop}
\begin{proof}
	We consider the map $\chi_g\colon\rr_g \rightarrow \mm_{2g-1}$ sending $[C,\eta]$ to the associated double cover $\widetilde{C}$ of $C$. Then we have 
	\[ \rr^r_g \subseteq \chi_g^{-1}\text{\large (}\mm^r_{2g-1,2g-2}\text{\large )}\]
	where $\cM^r_{2g-1,2g-2}$ is the locus of curves in $\cM_{2g-1}$ possessing a $g^r_{2g-2}$.  Consequently it is enough to show that $\mathrm{Im}(\chi_g\circ i)$ does not intersect $\mm^r_{2g-1,2g-2}$. But the image consists of curves as in the following figure: 
	\begin{figure}[H] \centering 
		\begin{tikzpicture}
			\draw [ xshift=4cm] plot [smooth, tension=1] coordinates { (0,0) (0.1,1) (0.2,2) (0.3,3)(0.4,4)(0.5,5)(0.6,6)(0.7,7)}
			node[left]{$R_1$};
			\draw [ xshift=4cm] plot [smooth, tension=1] coordinates { (4.5,0) (4.6,1) (4.7,2) (4.8,3)(4.9,4)(5,5)(5.1,6)(5.2,7)}
			node[left]{$R_2$};
			
			\draw[fill=black](4.05,0.5)circle(2pt);
			\draw[fill=black](4.25,2.5)circle(2pt);
			\draw[fill=black](4.55,5.5)circle(2pt);
			\node[above right = 0.5mm of {(4.05,0.5)}] {$x_1$};
			\node[above right = 0.5mm of {(4.25,2.5)}] {$x_2$};
			\node[above right = 0.5mm of {(4.55,5.5)}] {$x_{g-1}$};
			
			\draw[fill=black](4.64,6.45)circle(2pt);
		    \draw[fill=black](9.17,6.5)circle(2pt);
		    \node[above right = 0.5mm of {(4.64,6.45)}] {$y_1$};
		    \node[above right = 0.5mm of {(9.17,6.5)}] {$y_2$};
		    
		    \draw[fill=black](8.55,0.5)circle(2pt);
		    \draw[fill=black](8.75,2.5)circle(2pt);
		    \draw[fill=black](9.05,5.5)circle(2pt);
		    \node[above right = 0.5mm of {(8.55,0.5)}] {$x_g$};
		    \node[above right = 0.5mm of {(8.75,2.5)}] {$x_{g+1}$};
		    \node[above right = 0.5mm of {(9.05,5.5)}] {$x_{2g-2}$};

			\draw[fill=black](6.5,4.5)circle(1pt);
			\draw[fill=black](6.5,3.5)circle(1pt);
			\draw[fill=black](6.5,4)circle(1pt);
			
			\draw [ xshift=4cm] plot [smooth, tension=1] coordinates {(-3,0.3)(-2,0.5)(-1,0.3)(0,0.5)(1,0.4) }
			node[right]{$E_1$};
			
			\draw [ xshift=4cm] plot [smooth, tension=1] coordinates {(-2.8,2.3)(-1.8,2.5)(-0.8,2.3)(0.2,2.5)(1.2,2.4) }
			node[right]{$E_2$};
			
			\draw [ xshift=4cm] plot [smooth, tension=1] coordinates {(-2.5,5.3)(-1.5,5.5)(-0.5,5.3)(0.5,5.5)(1.5,5.4) }
			node[right]{$E_{g-1}$};
			
			\draw [ xshift=4cm] plot [smooth, tension=1] coordinates {(7.5,0.3)(6.5,0.5)(5.5,0.3)(4.5,0.5)(3.5,0.4) }
			node[left]{$E_{g}$};
			
			\draw [ xshift=4cm] plot [smooth, tension=1] coordinates {(7.7,2.3)(6.7,2.5)(5.7,2.3)(4.7,2.5)(3.7,2.4) }
			node[left]{$E_{g+1}$};
			
			\draw [ xshift=4cm] plot [smooth, tension=1] coordinates {(8,5.3)(7,5.5)(6,5.3)(5,5.5)(4,5.4) }
			node[left]{$E_{2g-2}$};

			\draw [ xshift=4cm] plot [smooth, tension=1] coordinates {(-1,6.4)(0,6.3)(1,6.5)(2,6.3)(3,6.4)(4,6.4)(5,6.5)(6,6.4)(7,6.5)}
			node[right]{$E_{2g-1}$};
			
		\end{tikzpicture}\caption{Element in the image of $\chi_g\circ i$}
		\end{figure}
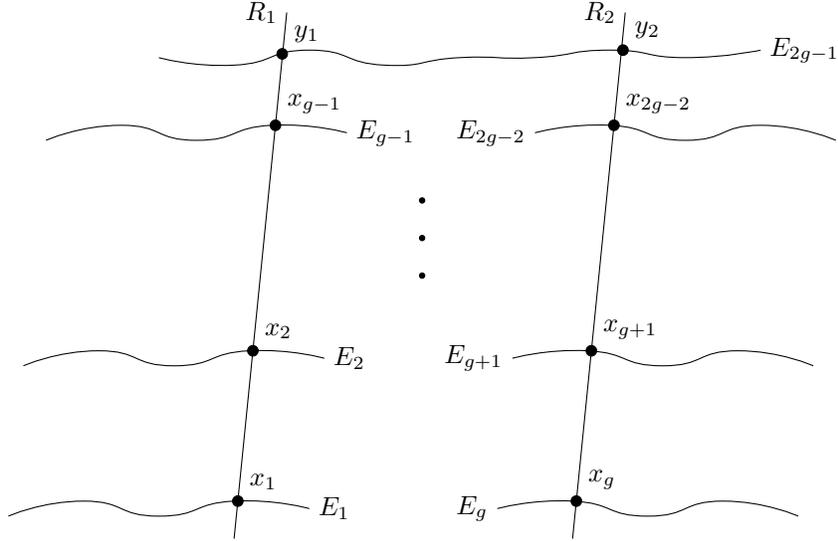
Here $R_1$ and $R_2$ are curves of genus 0, $[E_i,x_i]$ are copies of $[E,x]$ for $1\leq i \leq 2g-2$ and $[E_{2g-1}, y_1, y_2]$ is the double cover associated to $[E,x,\eta_E]$. 

We assume there is a curve in $\mathrm{Im}(\chi_g\circ i)$ admitting a limit $g^r_{2g-2}$. Because of I and II of the Lemma \ref{bnLemma}, the associated Brill-Noether number of all components, except the bridging elliptic curve, is greater or equal to 0. Because of part III in Lemma \ref{bnLemma}, the associated Brill-Noether number of the bridging elliptic curve is greater or equal to $-r$. 

Additivity of the Brill-Noether numbers imply 
\[ -r-2 = \rho(2g-1, r, 2g-2) \geq -r \]
As no curve in the image admits a limit $g^r_{2g-2}$ we deduce our conclusion. 
\end{proof}

Next, we consider $[X,x,\eta_X] \in \mathcal{C}^1\cR_{g-2}$ a generic Prym pointed curve and take the map 
\[ j\colon \mm_{2,1}\rightarrow \rr_g \]
sending a pointed curve $[C,y]$ to $ [X\cup_{x\sim y}C, \eta_X, \OO_C]$.
\begin{prop} \label{weierstrasspull} Let $\overline{\mathcal{W}} \subseteq \mm_{2,1}$ be the Weierstrass divisor. Then we have
	\[ j^*[\rr^r_g]  = c\cdot[\overline{\mathcal{W}}] \]
for some constant $c$. 
\end{prop}
\begin{proof} Because the Weierstrass divisor $\overline{\mathcal{W}}$ is irreducible, it is enough to show $j^{-1}(\rr^r_g) \subseteq    \overline{\mathcal{W}}$. Because $\rr^r_g \subseteq \chi_g^{-1}(\mm^r_{2g-1,2g-2})$, it is enough to show 
	\[ (\chi_g\circ j)^{-1}(\mm^r_{2g-1,2g-2}) \subseteq \overline{\mathcal{W}} \] 
But this follows as before using Proposition \ref{pointedSchwarz}, parts IV and V of Lemma \ref*{bnLemma} and additivity of the Brill-Noether numbers. 
\end{proof}

The map $j^*\colon \mathrm{Pic}(\rr_g) \rightarrow \mathrm{Pic}(\mm_{2,1})$ is described in \cite[Proposition 6.1]{BudKodPrym}.
\subsection{Pullbacks and Prym linear series}
Let $[C,\eta] \in \rr_g$ such that $C$ is of compact type and admits a unique irreducible component $X$ satisfying $\eta_X \ncong \OO_X$. For this component $X$, we denote by $p^X_1,\ldots, p^X_{s_X}$ its nodes and by $g^X_1,\ldots, g^X_{s_X}$ the genera of the connected components of $C\setminus X$ glued to $X$ at these points.  For an irreducible component $Y$ of $C$, different from $X$, we denote by $q^Y$ the node glueing $Y$ to the connected component of $C\setminus Y$ containing $X$. Let $p^Y_1,\ldots, p_{s_Y}^Y$ be the other nodes of $Y$. We denote by $g^Y_0, g^Y_1,\ldots, g^Y_{s_Y}$ the genera of the connected components of $C\setminus Y$ glued to $Y$ at these points. Furthermore, we denote by $\pi\colon\widetilde{C}\rightarrow C$ the double cover associated to $[C,\eta]$. With these notations set-up, we define the concept of a Prym limit $g^r_{2g-2}$: 
\begin{defi}
	In the notations above, a Prym limit $g^r_{2g-2}$ on $\pi\colon \widetilde{C}\rightarrow C$ is a crude limit $g^r_{2g-2}$ on $\widetilde{C}$ satisfying the following two conditions: 
	\begin{enumerate}
		\item For the unique component $\widetilde{X}$ of $\widetilde{C}$ above $X$, the $\widetilde{X}$-aspect $L_{\widetilde{X}}$ of the $g^r_{2g-2}$ satisfies
		\[Nm_{\pi_{|\widetilde{X}}} L_{\widetilde{X}} \cong \omega_X(\sum_{i=1}^s2g^X_ip_i)\]
		\item For a component $Y$ of $C$ different from $X$, we denote by $Y_1$ and $Y_2$ the two irreducible components of $\widetilde{C}$ above it. We identify these two components with $Y$ via the map $\pi$. With this identification, the $Y_1$ and $Y_2$-aspects of the $g^r_{2g-2}$ satisfy 
		\[ L_{Y_1}\otimes L_{Y_2} \cong \omega_Y((2g-2+2g^Y_0)q^Y + \sum_{i=1}^sg_i^Yp_i^Y) \]
	\end{enumerate} 
	
\end{defi}

It is immediate that we have the following lemma: 
\begin{lm}
	Let $[C,\eta]\in \rr^r_g$ with $C$ having a unique irreducible component $X$ for which $\eta_X \ncong \OO_X$. Let $\pi\colon\widetilde{C}\rightarrow C$ the double cover associated to $[C,\eta]$. Then $[\pi\colon \widetilde{C}\rightarrow C]$ admits a Prym limit $g^r_{2g-2}$. 
\end{lm}

Let $[E,p,\eta_E] \in \mathcal{C}^1\mathcal{R}_1$ generic and consider the map $\pi\colon \mm_{g-1,1} \rightarrow \rr_g$ sending a pointed curve $[Y,q]$ to the Prym curve $[Y\cup_{q\sim p}E,\OO_Y, \eta_E] \in \rr_g$. The pullback of this map at the level of Picard groups was computed in \cite[Proposition 6.1]{BudKodPrym}. We ask what divisors appear in the pullback $\pi^*(\rr^r_g)$. 

Let $g = \frac{r(r+1)}{2}$ and consider the sequence of vanishing orders $\textbf{a} = (0,2,\ldots, 2r)$. Let $\cM^r_{g-1,g+r-1}(\textbf{a})$ be the divisor in $\cM_{g-1,1}$ parametrizing pointed curves $[C,p]$ admitting a $g^r_{g+r-1}$ with vanishing orders greater or equal than $\textbf{a}$ at the point $p$. We have:
\begin{prop} \label{secondrel}Let $g = \frac{r(r+1)}{2}$ and $\textbf{a} = (0,2,\ldots,2r)$. Then there exists a constant $c$ such that at the level of divisorial classes we have: 
	\[ \pi^*(\rr^r_g) = c\cdot[\mm^r_{g-1,g+r-1}(\textbf{a})] + \Delta \] 
	where $\Delta$ is a boundary divisor that does not contain $\Delta_0$ in its support.
\end{prop}
\begin{proof}
	Let $[Y,q]$ be a generic element of a divisor in the pullback. Then the unramified double cover 
	\[ [Y_1\cup_{q_1\sim p_1}\widetilde{E}\cup_{p_2\sim q_2}Y_2\rightarrow Y\cup_{q\sim p}E]\]
	associated to $[Y\cup_{q\sim p}E, \OO_Y, \eta_E]$ admits a Prym limit $g^r_{2g-2}$. In particular $[Y_1\cup_{q_1\sim p_1}\widetilde{E}\cup_{p_2\sim q_2}Y_2]$ admits a limit $g^r_{2g-2}$ (with aspects denoted $L_{Y_1}, L_{Y_2}$ and $L_{\widetilde{E}}$) and we have the inequality 
	\[ \rho(2g-1,r,2g-2) = -r-2 \geq \rho(L_{Y_1}, q_1) + \rho(L_{Y_2}, q_2) + \rho(L_{\widetilde{E}}, p_1, p_2) \]
	Because $[Y,q]$ was chosen generic in a divisorial component we have $ \rho(L_{Y_1}, q_1) \geq -1$ and $ \rho(L_{Y_2}, q_2) \geq -1$. We also have from \cite[Proposition 1.4.1]{FarkasThesis} that $\rho(L_{\widetilde{E}}, p_1, p_2) \geq -r$ and hence all inequalities are actually equalities. 
	
	We denote by $0\leq a_0 <\cdots < a_r \leq d$ and $0\leq b_0 \coloneqq 2g-2-a_r <\cdots < b_r \coloneqq 2g-2-a_0 \leq d$ the ramification orders at $q_1$ and $q_2$ respectively. 
	
	Because $\rho(L_{Y_1}, q_1) = \rho(L_{Y_2}, q_2) = -1$ we have 
	\[ h^0\text{\large(}Y_1, L_{Y_1}(-a_iq_1)\text{\large)} = r+1-i \ \textrm{for every } \ 0\leq i \leq r \] 
	and 
	\[ h^0\text{\large(}Y_2, L_{Y_2}(-b_iq_2)\text{\large)} = r+1-i \ \textrm{for every } \ 0\leq i \leq r \]
	Using that $L_{Y_1}\otimes L_{Y_2} \cong \omega_Y(2gq)$ and $b_i = 2g-2-a_{r-i}$ we obtain  
	\[ h^0\text{\Large(}Y, \omega_Y\otimes L_{Y_1}^{-1}\text{\large(}(2+a_{r-i})q\text{\large)}\text{\Large)} = r+1-i \]
	This implies by Riemann-Roch: 
	\[ h^0\text{\Large(}Y_1, L_{Y_1}\text{\large(}-(2+a_{r-i})q\text{\large)}\text{\Large)} = g+r-1-a_{r-i}-i\]
	Using that $h^0\text{\Large(}Y_1, L_{Y_1}\text{\large(}-(2+a_r)q\text{\large)}\text{\Large)}\geq 0$ we get 
	\[ a_r \leq g+r-1\]
	Inverting the roles of the $a_i$'s and $b_i$'s we obtain 
	\[ a_0\geq g-r-1 \] 
	Because we have the divisorial equivalences on $\widetilde{E}$: 
	\[ a_ip_1 +b_{r-i}p_2 \equiv  a_jp_1 +b_{r-j}p_2 \]
	for every $0\leq i, j \leq r$, we get $a_i-a_{i-1} \geq 2$ for every $1\leq i \leq r$. This forces the unique possibility $a_i = g-r+2i-1$. Taking out the base locus $a_0q$ we get the conclusion. 
\end{proof}

Because of \cite{FarkasCastel}, computing the slope $\frac{a}{b_0'}$ in Theorem \ref{maintrm} becomes a purely combinatorial problem. Let 
\[ [\mathcal{W}_{g-1}] = \frac{g(g-1)}{2}\psi - \lambda - \sum_{i=1}^{g-2} \binom{g-i}{2} \delta_i  \]
and 
\[ [\mathcal{BN}_{g-1}] = (g+2)\lambda - \frac{g}{6}\delta_0 - \sum_{i=1}^{g-2} i(g-i-1)\delta_i\]
be the Weierstrass and Brill-Noether divisors in genus $g-1$. We know that $[\mm^r_{g-1,g+r-1}(\textbf{a})]$ is a linear combination of $[\mathcal{W}_{g-1}]$ and $[\mathcal{BN}_{g-1}]$. Let $\mu$ and $\nu$ the constants such that 
\[ [\mm^r_{g-1,g+r-1}(\textbf{a})] = \mu\cdot [\mathcal{BN}_{g-1}] + \nu \cdot[\mathcal{W}_{g-1}]\]
We know from \cite[Corollary 1 and Formula (6)]{FarkasCastel} that 
\[ \mu = -\frac{n}{2g(g-2)} + \frac{\Sigma}{2(g-2)(g-3)} \ \mathrm{and} \ \nu = \frac{n}{(g-2)(g-1)g}\]  
where 
\[ n = (g-1)!\cdot 2^{\frac{r(r+1)}{2}}\cdot \frac{(r-1)r^2(r+1)^2(r+2)}{16}\cdot \prod_{i=1}^r\frac{i!}{(2i)!} \]
and 
\[ \Sigma = (g-2)!\cdot 2^{\frac{r(r+1)}{2}}\cdot \prod_{j=1}^{r}\frac{j!}{(2j)!}\cdot \frac{1}{2^{2r-1}}\sum_{i=1}^r\frac{(2i)!(2r-2i+1)!}{(r-i)!\cdot (r-i)!\cdot i!\cdot (i-1)!}\cdot P(r,i) \]
where 
\[ P(r,i) = \frac{1}{16}(r^6+3r^5-21r^4-71r^3-100r^2-68r) + i(6r^3+12r^2+10r+4) -i^2(6r^2+6r+4) \]

Our next goal is to show that $ \Sigma = \frac{(g+1)(g-3)}{g(g-1)}\cdot n$. After simplifying, our goal is to show the equality 
\[ 2^{2r-6}\cdot \text{\large [}r(r+1)+2\text{\large]}(r-2)(r-1)r(r+1)(r+2)(r+3) = 2\sum_{i=1}^ri\cdot\binom{2i-1}{i}\cdot(r-i+1)\cdot\binom{2r-2i+1}{r-i}\cdot P(r,i)\]
This is an immediate consequence of the following three identities: 
\[ \sum_{i=1}^{r}i\cdot \binom{2i-1}{i}\cdot(r-i+1)\cdot \binom{2r-2i+1}{r-i} = \binom{r+1}{2}\cdot 2^{2r-2} \] 
\[ \sum_{i=1}^{r}i^2\cdot \binom{2i-1}{i}\cdot(r-i+1)\cdot \binom{2r-2i+1}{r-i} = 2^{2r-2}\cdot\binom{r+2}{3} + 2^{2r-3}\cdot \binom{r+1}{3}\]
and 
\[\sum_{i=1}^{r}i^3\cdot \binom{2i-1}{i}\cdot(r-i+1)\cdot \binom{2r-2i+1}{r-i} = 2^{2r-2}\cdot \binom{r+3}{4} + 5\cdot2^{2r-3}\cdot\binom{r+2}{4} + 2^{2r-4}\cdot\binom{r+1}{4}\]
These formulas can be obtained by looking at generating functions, by repeated derivations and multiplications, starting with the identity 
\[ \frac{1-\sqrt{1-4x}}{2x} = \sum_{i=0}^\infty \frac{(2i)!}{i!\cdot(i+1)!}x^i\]

As a consequence of this, we conclude that $\mu = \nu$.

Next, we consider the map $\pi_2\colon \mm_{g-1,2} \rightarrow \Delta_0''\subseteq \rr_g$ sending $[C,x,y]$ to $[C_{/x\sim y}, \eta]$ where $\eta$ satisfies $\nu^*\eta \cong \mathcal{O}_C$ for the normalization $\nu\colon C\rightarrow C_{/x\sim y}$. We ask what is the pullback of the divisor $\rr_g^r$ through this map. To answer this question we first define a divisor on $\cM_{g-1,2}$ and then show it is the pullback of $\cR_g^r$. 

Let $[C,x,y] \in \cM_{g-1,2}$ and consider the sequence $D_{\bullet}(x,y)$ of divisors
\[ 0 \leq x+y\leq \cdots \leq n(x+y) \leq \cdots   \]
together with the multivanishing sequence $\textbf{a} = (0,2,\ldots, 2r)$. We consider the locus
\begin{multline*} \cM^r_{g-1, g+r-1}\text{\large (}D_{\bullet}, \textbf{a}\text{\large )} \coloneqq \left\{ \right. [C,x,y] \in \cM_{g-1,2} \ | \ \exists \ (V,L)\in G^r_{g+r-1}(C) \ \textrm{satisfying} \ \\  h^0\text{\Large (}C,V(-i(x+y))\text{\Large )} \geq r+1-i \ \forall \ 0\leq i \leq r \left.\right\} 
\end{multline*}
parametrizing pointed curves having a $g^r_{g+r-1}$ with multivanishing orders with respect to $D_{\bullet}$ greater or equal to $\textbf{a}$.

Theorem \ref{stronglyBNdiv} implies that this locus is an irreducible divisor (except eventually for some higher codimension components).

\begin{prop} \label{sBNpull} \normalfont In the notations above, we have 
	\[\pi_2^*\rr^r_g =  c\cdot[\mm^r_{g-1, g+r-1}\text{\large (}D_{\bullet}, \textbf{a}\text{\large )}] + \Delta \] 
	for some constant $c$ and some boundary divisor $\Delta$ not containing $\Delta_0$ in its support. 
\end{prop}
\begin{proof}
	Let $[C,x,y] \in \cM_{g,2}$ generic in a divisorial component of $\pi_2^{-1}(\rr^r_g)$. We want to understand what properties such a $[C,x,y]$ must satisfy. We consider $[C_1, x_1, y_1]$ and $[C_2, x_2, y_2]$ two copies of $[C,x,y]$. The double cover associated to $\pi_2([C,x,y])$ is obtained by glueing together $x_1$ to $y_2$ and $y_1$ to $x_2$.
	
	We assume that the double cover admits a limit $g^r_{2g-2}$ respecting the norm condition. The same argument as in Proposition \ref{testA_1} implies that the multivanishing orders of the limit linear series are considered with respect to the sequences 
	\[ 0 \leq x_1+y_1\leq \cdots \leq g(x_1+y_1)\]
	and 
		\[ 0 \leq x_2+y_2\leq \cdots \leq g(x_2+y_2)\]
	We have two possibilities for the concentrated multidegree at $C_1$: either the multidegree is $(2g-3,1)$ or it is $(2g-2,0)$. 
	
	We assume the concentrated multidegree is $(2g-3,1)$. Then the $C_1$-aspect of the limit linear series is a $g^r_{2g-3}$ and let $0\leq a^1_0 \leq \cdots \leq a^1_r \leq 2g-4$ be its multivanishing orders with respect to $D_\bullet(x_1,y_1)$. Moreover, we denote by $r^1_l$ the number of times the value $2l$ appears in the sequence $a^1_0, \ldots a^1_r$. 
	
	Because $[C_1,x_1,y_1]$ is generic inside a divisor we have from Proposition \ref{rho-2} that: 
	\[ g-1 +(r+1)(2g-3-r-g+1) + \frac{r(r+1)}{2} -\sum_{j=0}^{r}a^1_j - \sum_{l=0}^{g-2} \binom{r^1_l}{2} \geq -1 \]
	Similarly we have: 
	\[ g-1 +(r+1)(2g-3-r-g+1) + \frac{r(r+1)}{2} -\sum_{j=0}^{r}a^2_j - \sum_{l=0}^{g-2} \binom{r^2_l}{2} \geq -1 \]
	But the compatibility condition implies $a^2_{r-i} + a^1_i \geq 2g-4$. It follows that 
	\[ (r+1)(2g-4) - \sum_{j=1}^{2}\sum_{l=0}^{g-2}\binom{r^j_l}{2} \geq \sum_{i=0}^{r}(a^1_i+a^2_i) \geq (r+1)(2g-4)\]
	and implicitly $r^1_l, r^2_l \in \left\{0,1\right\}$ for every $0\leq l\leq g-2$.
	
	We denote by $L_1$ and $L_2$ the $C_1$ and $C_2$ aspects of the limit linear series. Because $\rho\text{\large(}L, D_\bullet(x,y)\text{\large)} \leq -2$ cannot be a divisorial condition above $\mathcal{M}_{g-1,2}$, see Proposition \ref{rho-2}, it follows that 
	\[ h^0\text{\large(}L_1- b^1_i(x_1+y_1) \text{\large)} = r+1-i \ \forall\ 0 \leq i\leq r \] 
	and
		\[ h^0\text{\large(}L_2- b^2_i(x_2+y_2) \text{\large)} = r+1-i \ \forall \ 0\leq i\leq r \] 
   where we denoted $b^1_i \coloneqq \frac{a^1_i}{2}$ and $b^2_i \coloneqq \frac{a^2_i}{2}$. The norm condition implies: 
   \[ L_1\otimes L_2 \cong \omega_C\text{\large(}(g-1)(x+y) \text{\large )} \]
   It follows from here that 
   \[ L_1\text{\large(}b^1_i(x+y)\text{\large)} \cong \omega_C\otimes L_2^\vee\text{\large (}(g-1-b^1_i)(x+y) \text{\large )} \]
   and from Riemann-Roch we have 
   \[ h^0\text{\Large(}C, L_2\text{\large(}-(g-1-b^1_i)(x+y)\text{\large)}\text{\Large)} = r+2-g-i+2b^1_i \]
   But we know $a^1_i + a^2_{r-i} = 2g-4$, hence $g-1-b^1_i = b^2_{r-i} + 1$ and 
   \[ h^0\text{\Large(}C, L_2\text{\large(}-(b^2_{r-i}+1)(x+y)\text{\large)}\text{\Large)} = r+2-g-i+2b^1_i \]
   For $i = 0$ we obtain $2b^1_0 = g-r-2$. Similarly $2b_0^2 = g-r-2$, from where it follows $2b^1_r = g+r-2$. Because the $a^1_i$'s are all even and different, it follows that 
   \[ a^1_i = g-r-2+2i \ \forall \ 0\leq i \leq r \]
   We remark from the above computation that the case of admissible multidegree $(2g-3,1)$ is possible only when $g-r$ is even. The case of the admissible multidegree $(2g-2,0)$ is treated similarly, and in the end we will get 
   \[ a^1_i = a^2_i = g-r-1 + 2i \]
   making it feasible only when $g-r$ is odd. 
   
   As a conclusion, the pullback of $\rr^r_g$ is the irreducible divisor $\cM^r_{g-1, g+r-1}\text{\large (}D_{\bullet}, \textbf{a}\text{\large )}$. Lastly, we consider the composition map: 
   \[ \mm_{g-1,1} \rightarrow \mm_{g-1,2} \xrightarrow{\pi_2} \rr_g \]
   where the first map sends $[C,p]$ to $[C\cup_p \mathbb{P}^1, x,y]$ where $x, y\in \mathbb{P}^1$. The fact that the boundary divisor $\Delta_0$ is not contained in $\pi_2^{-1}(\rr^r_g)$ follows from Proposition \ref{testA_1}.
\end{proof}

\section{A strongly Brill-Noether divisor in $\mm_{g,2}$} \label{stronglysection}

For $d,r> 0$, we will study $g^r_d$'s respecting certain multivanishing conditions for a chain of divisors. Let $x, y$ two points on a curve $C$ of genus $g$ and $D_\bullet(x,y)$ a sequence of divisors as follows:
\[ 0 = D_0 < D_1\coloneqq d^1_1x + d^2_1y<\cdots <D_b \coloneqq  d^1_bx + d^2_by \]
Let $\textbf{a}$ be a sequence $0 \leq a_0 \leq a_1\leq\cdots a_r\leq d$ of multivanishing orders with respect to $D_\bullet(x,y)$. For $0\leq l \leq b-1$ we denote by $r_l$ the number of $a_i$'s equal to $\mathrm{deg}(D_l)$. 

We consider the locus $\cM^r_{g,d}(D_\bullet, \textbf{a})$ parametrizing $2$-pointed curves $[C,x,y]$ admitting a $g^r_d$ with multivanishing order at least $\textbf{a}$ for the divisorial sequence $D_\bullet(x,y)$. More concretely

\begin{multline*} \cM^r_{g, d}\text{\large (}D_\bullet, \textbf{a}\text{\large )} \coloneqq \left\{ \right. [C,x,y] \in \cM_{g,2} \ | \ \exists \ L\in W^r_d(C) \ \textrm{satisfying} \ \\  h^0\text{\Large (}C,L(-D_i)\text{\Large )} \geq r+1-\#\{a_j \ | \ a_j < \deg(D_i) \} \ \forall \ 0\leq i \leq r \left.\right\}.
\end{multline*}
When the expected codimension of this locus in $\cM_{g,2}$ is $2$ or higher, then it has no divisorial component:

\begin{prop} \label{rho-2} If 
	\[ \rho(g,r,d, D_\bullet, \textbf{a}) \coloneqq g-(r+1)(g-r+d)- \sum_{j=0}^{r}(a_j-j) - \sum_{l=0}^{b-1}\binom{r_l}{2} \leq -2 \]
	then every irreducible component of $\cM_{g,d}^r(D_\bullet, \textbf{a})$ has codimension at least $2$ in $\cM_{g,2}$.
\end{prop}
\begin{proof}
Let $[C\cup_p\mathbb{P}^1, x,y]$ an element of $\Delta_{0,\left\{1,2\right\}}$ contained in $\cM_{g,d}^r(D_\bullet, \textbf{a})$. Then $[C\cup_p\mathbb{P}^1, x,y]$ admits a limit $g^r_d$ with multivanishing orders greater or equal to $\textbf{a}$ for the sequence $D_\bullet(x,y)$. Because $[\mathbb{P}^1, p,x,y]$ is strongly Brill-Noether general and the Brill-Noether numbers are additive we have $\rho(C,p, L_C) \leq -2$. Because of Theorem 1.1 in \cite{IrreducibilityBN}, we deduce that: 
\begin{enumerate}
	\item The locus $\mm^r_{g,d}(D_\bullet, \textbf{a})$ is not equal to $\mm_{g,2}$ and 
	\item The locus $\mm^r_{g,d}(D_\bullet, \textbf{a})$ has no irreducible divisorial component intersecting $\Delta_{0,\left\{1,2\right\}}$.
\end{enumerate}
Consider now $[X,p] \in \cM_{g-1,1}$ generic and the map 
\[ \pi\colon \mm_{1,3} \rightarrow \mm_{g,2} \]
sending $[E,p,x,y]$ to $[X\cup_pE, x,y]$. 

Assume there exists $D$ a divisorial component of $\mm^r_{g,d}(D_\bullet, \textbf{a})$. Because $D\cap \Delta_{0,\left\{1,2\right\}} = \emptyset$,  we must have $\mathrm{Im}(\pi) \cap D \neq \emptyset$, otherwise $[D] = 0$ in $\mathrm{Pic}(\mm_{g,2})$. 

Let $[X\cup_pE,x,y] \in \mathrm{Im}(\pi)\cap \mm^r_{g,d}(D_\bullet, \textbf{a})$. Then it admits a limit $g^r_d$ with multivanishing sequence at least $\textbf{a}$ for $D_\bullet(x,y)$. The additivity of Brill-Noether numbers, together with the genericity of $[X,p]$ imply that 
\[ \rho(E, p, D_\bullet(x,y), L_E) \leq -2 \]
This implies that there exists $i<j$ such that 
\[ D_i(x,y) + (d-\mathrm{deg}D_i(x,y))p \equiv D_j(x,y) + (d-\mathrm{deg}D_j(x,y))p  \]
That can be rewritten 
\[ (d^1_j-d^1_i) x + (d^2_j-d^2_i)y \equiv (d^1_j-d^1_i +d^2_j-d^2_i) p \]
But, for any $a, b \in \mathbb{Z}_{\geq 0}$, every irreducible component of the locus in $\mm_{1,3}$ defined by 
\[ ax_2 +bx_3 \equiv (a+b)x_1 \]
intersects $\Delta_{0,\left\{2,3\right\}}$. This contradicts the condition $D\cap \Delta_{0,\left\{1,2\right\}} = \emptyset$. Hence, no such divisorial component $D$ exists and the conclusion follows.  

\end{proof}

Next, we restrict our attention to the case of Proposition \ref{sBNpull}; namely $g = \frac{r(r+1)}{2}-1$, the chain $D_\bullet(x,y)$ of divisors is 
\[ 0 \leq x+y\leq \cdots \leq n(x+y) \leq \cdots   \]
and the sequence $\textbf{a}$ of multivanishing orders is 
\[ \textbf{a}  = (a_0, a_1,\ldots, a_r)  = (0,2,\ldots, 2r). \]
Our goal is to prove 
 	\[\pi^*[\mm^r_{g, g+r}\text{\large (}D_{\bullet}, \textbf{a}\text{\large )}] = [\mm^r_{g,g+r}(\textbf{a})].\]
Our approach is to use intersection theory to prove this result. As such, we will need several classes in the cohomology of $C\times C\times \textrm{Pic}(C)$ and $C\times C\times C\times \textrm{Pic}(C)$. We consider the following classes 
\begin{enumerate}
	\item The class $\theta \in H^2 \text{\large(} \textrm{Pic}(C) \text{\large)} $, whose pullback to $C\times \cdots \times C \times \textrm{Pic}(C)$ will still be denoted by $\theta$. 
	\item A symplectic basis $\delta_1, \ldots, \delta_{2g}$ for $H^1(C, \mathbb{Z})\cong H^1(\textrm{Pic}^d(C), \mathbb{Z})$. Moreover, for a product $C\times \cdots \times C \times \textrm{Pic}^d(C)$, we denote by $\delta^i_{\alpha}$ the pullbacks of the symplectic basis via the projection map from $C\times \cdots \times C \times \textrm{Pic}^d(C)$ to its $i$-th entry. 
	\item Using the classes above, we define 
	\[ \gamma_{ij} \coloneqq - \sum_{i=1}^g \text{\Large(} \delta^j_s\delta^i_{g+s}-\delta^j_{g+s}\delta^i_s\text{\Large )}\]
    \item Lastly, by pulling back the class of a point in $C$ via the $i$-th projection map, we obtain the class $\eta_i$ in $H^2\text{\Large (}C\times \cdots \times C \times \textrm{Pic}^d(C)\text{\Large )}$.
\end{enumerate} 
With this notation set, we are ready to prove our result: 
 \begin{prop} \label{FTsimilar} \normalfont
 	In the notation above, we consider the pullback 
 	\[ \pi\colon \mm_{g,1}\rightarrow \mm_{g,2} \]
 	sending $[X,p]$ to $[X\cup_p\mathbb{P}^1,x,y]$ where $[\mathbb{P}^1,p,x,y]$ is the unique curve in $\cM_{0,3}$. Then we have  
 	\[\pi^*[\mm^r_{g, g+r}\text{\large (}D_{\bullet}, \textbf{a}\text{\large )}] = [\mm^r_{g,g+r}(\textbf{a})]\]
  \end{prop}
\begin{proof} Looking at limit linear series respecting the multivanishing condition, it is clear that the only divisorial component in the preimage $\pi^{-1}\mm^r_{g, g+r}\text{\large (}D_{\bullet}, \textbf{a}\text{\large )}$ is $\mm^r_{g,g+r}(\textbf{a})$. Hence the only thing we need to show is that its multiplicity is $1$. For this we restrict our attention to a unique smooth curve $C$ and consider the locus: 
	\[ Z = \left\{(x,y,L)\in C\times C\times \mathrm{Pic}^{g+r}(C) \ | \ h^0\text{\Large(}C, L\text{\large(}-i(x+y)\text{\large)}\text{\Large)}\geq r+1-i \ \forall \ 0\leq i \leq r\right\} \]
Let $\Delta$ be the diagonal of $C\times C$ and $n_{g,r,g+r, \textbf{a}}$ the degree of $\mathcal{G}^r_{g,g+r}(\textbf{a})$ over $\mathcal{M}_g$ as computed in \cite{FarkasCastel}. Because the locus $\mathcal{G}^r_{g,g+r}(\textbf{a})$ has a unique irreducible component dominating $\mathcal{M}_g$, see \cite[Theorem 1.2]{IrreducibilityBN}, the proposition follows if we prove that 
\[ [Z]\cdot [\Delta\times \mathrm{Pic}^{g+r}(C)] = n_{g,r,g+r, \textbf{a}}. \]
Let $p \in C$ be a general point and $m > g-1-r$. Then, the condition $h^0\text{\Large(}C, L\text{\large(}-i(x+y)\text{\large)}\text{\Large)} \geq r+1-i$ can be rewritten as 
\[ \mathrm{rk}\text{\Large (}H^0\text{\large(}C, L(mp)\text{\large)} \rightarrow H^0\text{\large(}C, L(mp)_{|mp+ix+iy}\text{\large)}\text{\Large )} \leq m+i. \]
This globalizes to a map of vector bundles over $C\times C\times \mathrm{Pic}^{g+r}(C)$:
\[ \pi^*\mathcal{E} \rightarrow \mu_*(\nu^*\mathcal{L}\otimes \OO_{D_i})\eqqcolon \mathcal{M}_i.\] 
Here we denoted: 
\begin{itemize}
	\item $\mathcal{L}$ a Poincar\'e bundle on $C\times \mathrm{Pic}^{g+r+m}(C)$, 
	\item $\mathcal{E}$ the pushforward of $\mathcal{L}$ to $\mathrm{Pic}^{g+r+m}(C)$, 
	\item $ \pi\colon C\times C \times \mathrm{Pic}^{g+r+m}(C) \rightarrow \mathrm{Pic}^{g+r+m}(C)$ the projection onto the third factor, 
	\item $\nu \colon C\times C\times C \times \mathrm{Pic}^{g+r+m}(C) \rightarrow C\times \mathrm{Pic}^{g+r+m}(C)$ the projection onto the first and fourth factors, 
	\item $\mu\colon C\times C\times C \times \mathrm{Pic}^{g+r+m}(C) \rightarrow C\times C \times \mathrm{Pic}^{g+r+m}(C)$ the projection onto the second, third and fourth factors,
	\item $D_i$ is the pullback to $C\times C\times C \times \mathrm{Pic}^{g+r+m}(C)$ of the divisor in $C\times C\times C$ defined as
	\[ D_i \coloneqq m\cdot \left\{p\right\}\times C\times C + i\Delta_{12} + i\Delta_{13}.\]
\end{itemize}

In particular 
\begin{align*}
	[D_i] &= m\eta_1 +(i\eta_1 +i\eta_2 +i\gamma_{12}) + (i\eta_1 +i\eta_3 +i\gamma_{13}) \\
	&= (m+2i)\eta_1 +i\eta_2 +i\eta_3 + i\gamma_{12} + i\gamma_{13}.
\end{align*}  

We have the maps of vector bundles 
\[ \pi^*\mathcal{E} \rightarrow \mathcal{M}_r \twoheadrightarrow \cdots \twoheadrightarrow \mathcal{M}_0 \]
and $Z$ the degeneracy locus. We can compute its class $[Z]$ using the Fulton-Pragacz determinantal formula for flag bundles, see \cite{FultonPragaczformula}. 

We have from \cite[Chapter VIII]{ACGC1} that 
\begin{align*}
	ch(\nu^*\mathcal{L}) =& 1 + (g+r+m) \eta_1 + \gamma_{14}-\eta_1\theta \\ 
	ch(\mathcal{O}_{D_i}) =& 1 - e^{-(m+2i)\eta_1-i\eta_2 - i\eta_3 -i\gamma_{12} -i\gamma_{13}}
\end{align*}
From these we can compute $ch(\nu^*\mathcal{L}\otimes \OO_{D_i}) = ch(\nu^*\mathcal{L})\cdot ch(\OO_{D_i})$. Via the Grothendieck-Riemann-Roch formula we have 
\[ ch(\mathcal{M}_i) = \mu_{*}\text{\Large (}\text{\large(}1+(1-g)\eta_1\text{\large)}\cdot ch(\nu^*\mathcal{L}\otimes \OO_{D_i})\text{\Large )}. \]

From this, we obtain the Chern classes of $\mathcal{M}_i$:
\begin{align*}
	ch_0(\mathcal{M}_i)\ = & \ m+2i \\
	ch_1(\mathcal{M}_i)\ = & \ i(r+1+ig-2i)(\eta_2 + \eta_3) + i(\gamma_{24}+\gamma_{34}) - i^2\gamma_{23} \\
	ch_2(\cM_i) \ = & \ -i(\eta_2+\eta_3)\theta + i^2(2i-r-1-2ig)\eta_2\eta_3 - i^2(\eta_2\gamma_{34} + \eta_3 \gamma_{24}) \\ 
	ch_3(\cM_i) \ = & \ i^2 \eta_2\eta_3\theta 
\end{align*}

We can compute the Chern classes of $\cM_i$ by knowing the Chern character and obtain: 
\begin{align*}
	c_2(\cM_i) \ = & \ i^2\text{\Large [}(r+1+ig-2i)^2 - gi^4 + 2ig +r+1-2i \text{\Large ]} \eta_2 \eta_3 + i^2(r+2+ig-3i)(\eta_2\gamma_{34}+ \eta_3\gamma_{24}) \\ &+(i-i^2)(\eta_2 +\eta_3)\theta + i^2 \gamma_{24}\gamma_{34} \\
	c_3(\cM_i) \ = & \ i^2\text{\Large [} 2(1-i)(r+1+ig-2i) - 4i +2i^2 +1 \text{\Large ]}\eta_2\eta_3\theta
\end{align*}

The Fulton-Pragacz formula gives the class of our locus $Z$ as a determinant with entries of the form $c_j(\cM_i - \mathcal{E})$. Our goal is to compute the intersection of this class with $[\Delta\times \mathrm{Pic}^{g+r}(C)] = \eta_2 + \gamma_{23} +\eta_3$. Consequently, any class that vanishes when multiplied with $\eta_2$, $\gamma_{23}$ and $\eta_3$ is irrelevant for our computation. Hence we can work with the numerically simplified classes below and still get the desired result: 
\begin{align*}
	c_0'(\cM_i) \ \coloneqq & \ c_0(\cM_i) = 1 \\ 
	c_1'(\cM_i) \ \coloneqq & \ c_1(\cM_i) = i(r+1+ig-2i) \eta +i\gamma - i^2 \gamma_{23} \\
	c_2'(\cM_i) \ \coloneqq & \ (i-i^2)\eta\theta + i^2 \gamma_{24}\gamma_{34} \\ 
	c_3'(\cM_i) \ \coloneqq & \ 0
\end{align*} 
Here we denoted $\eta \coloneqq \eta_2 + \eta_3 $ and $\gamma \coloneqq \gamma_{24} + \gamma_{34}$. Using the class \[c_t'(\cM_i-\pi^*\mathcal{E}) \coloneqq c_t'(\cM_i) \cdot c_t(-\pi^*\mathcal{E}) = c_t'(\cM_i)\cdot e^{t\theta}\]
we have 
\begin{align*}
	c_1^{(i)} \ =& \ \text{\Large [}i^2(g-2) + i(r+1)\text{\Large ]}\eta + i\gamma - i^2\gamma_{23} + \theta \\ 
	c_j^{(i)} \ = & \ \frac{\theta^j}{j!} + \text{\LARGE [} \frac{i^2(g-2) + i(r+1)}{(j-1)!} + \frac{i-i^2}{(j-2)!} \text{\LARGE ]}\eta\theta^{j-1} + \frac{i}{(j-1)!}\gamma\theta^{j-1} - \frac{i^2}{(j-1)!}\gamma_{23}\theta^{j-1} + \frac{i^2}{(j-2)!}\gamma_{24}\gamma_{34}\theta^{j-2}
\end{align*}
By the Fulton-Pragacz formula, the intersection we want to compute is equal to 
\[ (\eta_2 + \gamma_{23} +\eta_3)\cdot \mathrm{det}(c^{(r+1-i)}_{r+1-2i+j})_{1\leq i,j\leq r+1}\]
This can be rewritten as  
\[(\eta_2 + \gamma_{23} +\eta_3)\cdot \mathrm{det}(c^{(i)}_{2i-j})_{0\leq i,j\leq r}\]
We consider the $\theta$-pure part of the matrix, that is, we take as the $(i,j)$ entry the coefficient of $\theta^{2i-j}$ in $c^{(i)}_{2i-j} $. We obtain in this way the matrix $(\frac{1}{(2i-j)!})_{0\leq i,j\leq r}$. 

We recall from \cite{FarkasCastel} that  \[\mathrm{det}(\frac{1}{(b_i-j)!})_{0\leq i, j\leq r} = \frac{\prod_{l<k}(b_k-b_l)}{\prod_{j=0}^{r}b_j!}\eqqcolon V(b_0,b_1,\ldots, b_r)\]
Because $\eta\cdot \gamma_{23} = \eta^3 = \gamma_{24}\gamma_{34}\eta = 0$, there are only two ways to obtain non-zero terms when computing $\eta \cdot \mathrm{det}(c^{(i)}_{2i-j})_{0\leq i,j\leq r}$, namely: 
\begin{itemize}
	\item In the determinant, multiply $r$ summands $\frac{\theta^j}{j!}$ with a summand of the form 
	\[ \text{\LARGE [} \frac{i^2(g-2) + i(r+1)}{(j-1)!} + \frac{i-i^2}{(j-2)!} \text{\LARGE ]}\eta\theta^{j-1} \]
	\item  In the determinant, multiply $r-1$ summands $\frac{\theta^j}{j!}$ with two summands of the form 
	\[ \frac{i}{(j-1)!}\gamma\cdot \theta^{j-1} \] 
\end{itemize}  
The first possibility produces a contribution of 
\[ 2\sum_{i=0}^r\text{\Large [}i^2(g-2) +i(r+1) \text{\Large ]}\cdot V(0,2,\ldots, 2i-2, 2i-1, 2i+2,\ldots, 2r) \]
where the sequence has $2k$ as the $(k+1)$-th entry if $k\neq i$ and $2i-1$ on the $(i+1)$-th position. The contribution coming from the second possibility is 
\[ - 4\sum_{0\leq i_1<i_2\leq r}i_1i_2\cdot V(0,2,\ldots, 2i_1-1, \ldots, 2i_2-1,\ldots, 2r)\]
where the sequence has $2k$ as the $(k+1)$-th entry if $k\neq i_1$ or $i_2$; and the entry is one less than that if $k =i_1$ or $i_2$.

Next, we compute $\gamma_{23}\cdot \mathrm{det}(c^{(i)}_{2i-j})_{0\leq i,j\leq r}$. Because $\gamma_{23}\cdot \eta = \gamma^2\cdot \gamma_{23} = 0$, there are three possible ways to obtain a non-zero contribution: 
\begin{itemize}
	\item In the determinant, multiply $r$ summands $\frac{\theta^j}{j!}$ with a summand of the form  
	\[  - \frac{i^2}{(j-1)!}\gamma_{23}\theta^{j-1}\]
	\item In the determinant, multiply $r-1$ summands $\frac{\theta^j}{j!}$ with two summands of the form
	\[\frac{i}{(j-1)!}\gamma\theta^{j-1}\]
	\item In the determinant, multiply $r$ summands $\frac{\theta^j}{j!}$ with a summand of the form  
	\[  \frac{i^2}{(j-2)!}\gamma_{24}\gamma_{34}\theta^{j-2}\]
\end{itemize}
We observe immediately that the third possibility does not contribute to the result because the associated Vandermonde determinant is 0. The first possibility gives a contribution of 
\[ \sum_{i=0}^r 2gi^2\cdot V(0,2,\ldots, 2i-2, 2i-1, 2i+2,\ldots, 2r) \]
while the second possibility contributes 
\[ - 4\sum_{0\leq i_1<i_2\leq r}i_1i_2\cdot V(0,2,\ldots, 2i_1-1, \ldots, 2i_2-1,\ldots, 2r)\]
to the result. Here the sequences are considered as in the computation of $\eta \cdot \mathrm{det}(c^{(i)}_{2i-j})_{0\leq i,j\leq r}$. 

Hence, the intersection $[Z] \cdot [\Delta\times \mathrm{Pic}^{g+r+m}]$ is equal to 

\begin{align*}
	&2\sum_{i=0}^r\text{\Large [}i^2(g-2) +i(r+1) +i^2g\text{\Large ]}\cdot V(0,2,\ldots, 2i-2, 2i-1, 2i+2,\ldots, 2r) \\
	& - 8\sum_{0\leq i_1<i_2\leq r}i_1i_2\cdot V(0,2,\ldots, 2i_1-1, \ldots, 2i_2-1,\ldots, 2r)
\end{align*} 
This is just the formula for $n_{g,r,g+r,\textbf{a}}$ appearing in \cite[Equality (5)]{FarkasCastel}, hence we are done.
\end{proof}

We are now ready to prove Theorem \ref{stronglyBNdiv}. 

\textbf{Proof of Theorem \ref{stronglyBNdiv}:} The same method as in Proposition \ref{rho-2} implies that every irreducible divisorial component of $\mm^r_{g, g+r}\text{\large (}D_{\bullet}, \textbf{a}\text{\large )}$ intersects the boundary divisor $\Delta_{0,\left\{1,2\right\}}$. But Proposition \ref{FTsimilar} implies that $\mm^r_{g, g+r}\text{\large (}D_{\bullet}, \textbf{a}\text{\large )}\cap\Delta_{0,\left\{1,2\right\}}$ is reduced and irreducible. This implies irreducibility of  $\mm^r_{g, g+r}\text{\large (}D_{\bullet}, \textbf{a}\text{\large )}$. 

Using 	
\[\pi^*[\mm^r_{g, g+r}\text{\large (}D_{\bullet}, \textbf{a}\text{\large )}] = [\mm^r_{g,g+r}(\textbf{a})]\]
together with the discussion following Proposition \ref{secondrel}, we deduce $a = g+2$, $b_0 = \frac{g+1}{6}$, $b_{i,\left\{1,2\right\}} = \frac{(g-i)(g+i+1)}{2}$ and $c = \frac{(g+1)!}{g-1}\cdot 2^{g-1}\prod_{i=1}^{r}\frac{i!}{(2i)!}$. That $a_1 = a_2$ is obvious from the symmetry of the situation. 

To conclude that $ a_1 = a_2 = \frac{g^2+g+2}{8}$ we consider a generic curve $[C,y] \in \cM_{g,1}$ and take the test curve $A \coloneqq \left\{[C,x,y] \right\}_{x\in C}$. We know that $A\cdot \psi_1 = (2g-1)$, $A\cdot \psi_2 = 1$, $A\cdot \delta_{0,\left\{1,2\right\}} = 1$ while the intersection of $A$ with $\lambda$ and all other boundary classes is $0$. 

It is a consequence of Proposition \ref{FTsimilar} that 
\[ c\cdot\text{\large[}(2g-1)a_1+a_2-b_{0,\left\{1,2\right\}}\text{\large]} = \eta_2\cdot\mathrm{det}(c^{(i)}_{2i-j})_{0\leq i,j\leq r} \]
This can be rewritten as 
\begin{align*}
	c\cdot \text{\large[}2ga_1 - b_{0,\left\{1,2\right\}} \text{\large]} =& \sum_{i=0}^r\text{\large[}i^2(g-2) +i(r+1)\text{\large]}\cdot  V(0,2,\ldots, 2i-2, 2i-1, 2i+2,\ldots, 2r)- \\
	& - 2\sum_{0\leq i_1<i_2\leq r}i_1i_2\cdot V(0,2,\ldots, 2i_1-1, \ldots, 2i_2-1,\ldots, 2r)
\end{align*}
We want to prove that $a_1 = \frac{g^2+g+2}{4g(g+1)}\cdot b_{0,\left\{1,2\right\}}$. Hence, what we need to show is that 
\begin{align*}
	c\cdot \frac{g^2-g}{2(g+1)}\cdot b_{0,\left\{1,2\right\}} =& \sum_{i=0}^r\text{\large[}i^2(g-2) +i(r+1)\text{\large]}\cdot  V(0,2,\ldots, 2i-2, 2i-1, 2i+2,\ldots, 2r)- \\
	& - 2\sum_{0\leq i_1<i_2\leq r}i_1i_2\cdot V(0,2,\ldots, 2i_1-1, \ldots, 2i_2-1,\ldots, 2r)
\end{align*}

We denote $n \coloneqq g!\cdot 2^{\frac{r(r+1)}{2}}\cdot\prod_{i=1}^r\frac{i!}{(2i)!}\cdot \frac{(r-1)r^2(r+1)^2(r+2)}{16}$ and note that $b_{0,\left\{1,2\right\}} = \frac{n}{2g-2}$. We want to show that 
\begin{align*}
	\frac{1}{4}\cdot n - \frac{1}{4(g+1)}\cdot n =& \sum_{i=0}^r\text{\large[}i^2(g-2) +i(r+1)\text{\large]}\cdot  V(0,2,\ldots, 2i-2, 2i-1, 2i+2,\ldots, 2r)- \\
	& - 2\sum_{0\leq i_1<i_2\leq r}i_1i_2\cdot V(0,2,\ldots, 2i_1-1, \ldots, 2i_2-1,\ldots, 2r)
\end{align*}

Formula $(5)$ in \cite{FarkasCastel} implies 
\begin{align*}
	\frac{1}{4}\cdot n  =& \sum_{i=0}^r\text{\large[}i^2(g-1) +\frac{i(r+1)}{2}\text{\large]}\cdot  V(0,2,\ldots, 2i-2, 2i-1, 2i+2,\ldots, 2r)- \\
	& - 2\sum_{0\leq i_1<i_2\leq r}i_1i_2\cdot V(0,2,\ldots, 2i_1-1, \ldots, 2i_2-1,\ldots, 2r)
\end{align*}
Hence, we want to show that 
\[ \frac{1}{4(g+1)}\cdot n = \sum_{i=0}^r \text{\Large[}i^2-\frac{i(r+1)}{2}\text{\Large]}\cdot  V(0,2,\ldots, 2i-2, 2i-1, 2i+2,\ldots, 2r) \]
After substituting $g = \frac{r(r+1)}{2}-1$ and simplifying common terms, we are left to show 
\[ \frac{(r-1)r(r+1)(r+2)}{16}\cdot 2^{2r-1} = \sum_{i=0}^r\text{\Large[}4i^3-2i^2(r+1)\text{\Large]}\binom{2i-1}{i}\cdot (r-i+1)\binom{2r-2i+1}{r-i} \]
The formulas in the discussion following Proposition \ref{secondrel} imply the conclusion. \hfill$\square$



\section{The Kodaira dimension of $\cR_{14,2}$} 

We will now see that Theorem \ref{maintrm} is just a numerical consequence of the results in the previous sections. Furthermore, we will use the divisor $\rr^5_{15}$ to conclude Theorem \ref{Kodaira}. 

\textbf{Proof of Theorem \ref{maintrm}:}
Let $[\rr^r_g] = a\lambda - b_0'\delta_0'- b_0''\delta_0''-b_0^{\mathrm{ram}}\delta_0^{\mathrm{ram}} - \sum\limits_{i=1}^{\floor{\frac{g}{2}}} (b_i\delta_i + b_{i:g-i} \delta_{i:g-i} + b_{g-i}\delta_{g-i})$. We know from Proposition \ref{testA_g-1} and Proposition \ref{testA_1} that
\[ a-12b_0'+b_{g-1}= a-4b_0'' - 4b_0^{\mathrm{ram}} + b_1 = 0 \]
Furthermore, we get as a consequence of Proposition \ref{Picardmap}, Proposition \ref{zeropullback} and Proposition \ref{weierstrasspull} that
\[ b_i = \frac{(i-1)(g-i)}{g-2}b_{g-1} + \frac{(g-i-1)(g-i)}{(g-1)(g-2)}b_1 \ \textrm{for} \ 1\leq i\leq g-1 \]
and
\[ b_{g-2} = 30b_0'-3a\]

Proposition \ref{secondrel} implies $\frac{a}{b_0'} = 6 + \frac{6}{g}$. Similarly, Proposition \ref{sBNpull} and Theorem \ref{stronglyBNdiv} imply $\frac{a}{b_0''} = \frac{8g+8}{g^2-g+2}$. 

These relations are sufficient to conclude Theorem \ref{maintrm}. \hfill $\square$

Theorem \ref{maintrm} checks out for $r = 3, g= 6$. In fact, for the divisor $\rr^3_6 = \overline{\mathcal{Q}}$ on $\rr_{6}$ appearing in \cite{FarkasGrushevsky} we have: 
\[ \overline{\mathcal{Q}} = 7\lambda - \delta_0' -4\delta_0''-\frac{3}{2}\delta_0^{\textrm{ram}} - 15\delta_1 - 5\delta_5- 14\delta_2- 9\delta_4 -12 \delta_3 - \cdots \]
where the missing coefficients are still unknown. 

The values 
\[ \frac{a}{b_0'} = 6+ \frac{6}{g}, \ \frac{a}{b_0''} = \frac{8g+8}{g^2-g+2} \ \textrm{and} \ \frac{a}{b_0^{\textrm{ram}}} = 4 + \frac{4}{g} \]
are referred as the slopes of the divisor and are usually important in understanding the birational geometry of moduli spaces. Divisors having small slopes are particularly useful for expressing the canonical divisor as a sum of an effective and an ample divisor. This is enough to conclude that the respective moduli space is of general type. We can use the Prym-Brill-Noether divisor to prove Theorem \ref{Kodaira}. 

\textbf{Proof of Theorem \ref{Kodaira}:} We consider the map 
\[ \chi_{14,2}\colon \rr_{14,2}\rightarrow \mm_{28} \] 
sending a tuple $[C,x+y,\eta]$ to the associated double cover $\widetilde{C}$, and the map 
\[ \iota_{14,2}\colon \rr_{14,2}\rightarrow \Delta_0^{\mathrm{ram}}\subseteq \rr_{15}\]
sending $[C,x+y,\eta]$ to $[C\cup_{x,y}R, \eta, \OO_R(1)]$, that is, glueing a rational component at the points $x, y$ and considering a line bundle that restricts to $\eta$ on $C$ and to $\OO_R(1)$ on $R$. 

On $\rr_{14,2}$ we consider the following divisors (up to multiplication with a constant): 
\begin{enumerate}
	\item The pullback to $\rr_{14,2}$ of the Brill-Noether divisor on $\mm_{14}$: 
	\[ [\mathcal{BN}_{14}] = 34\lambda - 5\delta'_0 - 10\delta_0^{\mathrm{ram}}-\cdots \]
	\item The pullback to $\rr_{14,2}$ of the Gieseker-Petri divisor, see \cite[]{Farkosz}, $\overline{\mathcal{GP}}^3_{28,24}$, whose pullback has class 
	\[\chi^*_{14,2}[\overline{\mathcal{GP}}^3_{28,24}] = 19289\psi + 308624\lambda - 47784\delta_0' - 62470\delta_0^{\mathrm{ram}} - \cdots \]
	\item The pullback to $\rr_{14,2}$ of the Prym-Brill-Noether divisor on $\rr_{15}$: 
	\[ \iota_{14,2}^*[\rr^5_{15}] = 15\psi +128\lambda - 20\delta_0' - 30\delta_0^{\mathrm{ram}}-\cdots \]
\end{enumerate}
	We consider the linear combination 
\[ \frac{4603}{63570}\cdot [\mathcal{BN}_{14}] + \frac{1}{50856}\cdot \chi^*_{14,2}[\overline{\mathcal{GP}}^3_{28,24}] + \frac{683}{19560}\cdot  \iota_{14,2}^*[\rr^5_{15}] \]
which is equal to 
\[\text{\Large(}\frac{19289}{50856}+\frac{15\cdot683}{19560}\text{\Large)}\psi +13\lambda - 2\delta_0' - 3\delta_0^{\mathrm{ram}}-\cdots \]
Because $\frac{19289}{50856}+\frac{15\cdot683}{19560} < 1$ and $\psi$ is big and nef, the conclusion follows as in \cite{BudKodPrym}. All the other slope conditions are obviously satisfied. \hfill $\square$

\section{Curves on Nikulin surfaces}

A polarized Nikulin surface of genus $g$ is a smooth polarized $K3$ surface $(S, H)$ equipped with a double cover $f\colon \widetilde{S} \rightarrow S$ branched along eight disjoint rational curves $N_1, \ldots, N_8$ such that $N_i\cdot H = 0$ for all $1\leq i \leq 8$. 

Denoting by $e \in \textrm{Pic}(S)$ the class defined by the relation 
\[ e^{\otimes 2} \cong \OO_S(\sum_{i = 1}^8 N_i)\]
we consider the Nikulin lattice 
\[ \mathfrak{R} \coloneqq \langle e, \OO_S(N_1), \ldots, \OO_S(N_8)\rangle.\]

Out of this lattice, we obtain a primitive embedding $j\colon \Lambda_g \coloneqq \mathbb{Z}\cdot[H] \oplus \mathfrak{R} \rightarrow \textrm{Pic}(S)$ and use it to study Nikulin surfaces. As a consequence, the moduli space $\mathcal{F}^{\mathfrak{R}}_g$ of genus $g$ Nikulin surfaces is an $11$-dimensional subspace of the moduli space of polarized $K3$ surfaces, see \cite{Dolgachev-Nikulin} and \cite{Nikulin}. We can endow this with a tautological $\mathbb{P}^g$ bundle 
\[\mathcal{P}^{\mathfrak{R}}_g \coloneqq \left\{ [S,j\colon \Lambda_g\hookrightarrow \textrm{Pic}(S), C] \ | \ C \in |H| \ \textrm{is a smooth curve of genus} \ g \right\}. \]
This bundle comes equipped with a projection map
\[ \mathcal{P}^{\mathfrak{R}}_g \rightarrow \cR_g \]
sending a tuple $[S,j,C]$ to $[C, e_C \coloneqq e\otimes \OO_C]$. Moving a generic such curve $C$ in a pencil of $|H|$, we obtain a test curve $\Xi_g$ whose intersection with all divisorial classes of $\rr_g$ is known, see \cite[Proposition 1.4]{FarVerNikulintheta}. We have 
\[ \Xi_g\cdot \lambda = g+1, \ \Xi_g\cdot \delta_0'=6g+2, \ \Xi_g\cdot \delta_0^{\textrm{ram}} = 8 \]
while the intersection with all other boundary divisors is $0$. 

For $g = \frac{r(r+1)}{2}$, we immediately compute  
\[ \Xi_g \cdot [\rr^r_g] = c\cdot (1-\frac{g}{3}). \]
As soon as $r \geq 3$, this number is negative, and hence the Nikulin locus is contained in the Prym-Brill-Noether divisor. This concludes Corollary \ref{nikulin}.
\bibliography{main}

\begin{thebibliography}{FGSMV14}

\bibitem[ACGH85]{ACGC1}
E.~Arbarello, M.~Cornalba, P.~Griffiths, and J.~Harris.
\newblock {\em Geometry of algebraic curves. {V}olume {I}}.
\newblock Grundlehren der Mathematischen Wissenschaften. Springer, Heidelberg,
  1985.

\bibitem[BCF04]{Casa}
E.~Ballico, C.~Casagrande, and C.~Fontanari.
\newblock Moduli of {P}rym curves.
\newblock {\em Documenta Mathematica}, \textbf{9}:265--281, 2004.

\bibitem[Bea77]{Beau77}
A.~Beauville.
\newblock Vari\'et\'es de {P}rym et jacobiennes interm\'ediaires.
\newblock {\em Annales scientifiques de l'\'Ecole Normale Sup\'erieure}, 4e
  s{\'e}rie, \textbf{10}:309--391, 1977.

\bibitem[Ber87]{Bertram}
A.~Bertram.
\newblock An existence theorem for {P}rym special divisors.
\newblock {\em Inventiones mathematicae}, \textbf{90}:669--671, 1987.

\bibitem[Bru16]{Bruns}
G.~Bruns.
\newblock {$\overline{\mathcal{R}}_{15}$} is of general type.
\newblock {\em Algebra {\normalfont \&} Number Theory}, \textbf{10}:1949--1964,
  2016.

\bibitem[Bud21]{Bud-adm}
A.~Bud.
\newblock A {H}urwitz divisor on the moduli of {P}rym curves.
\newblock {\em Geometriae Dedicata}, \textbf{216}, 2021.

\bibitem[Bud22a]{BudPrymIrr}
A.~Bud.
\newblock Irreducibility of a universal {P}rym-{B}rill-{N}oether locus.
\newblock {\em International Mathematics Research Notices},
  \textbf{2023}:10174--10180, 2022.

\bibitem[Bud22b]{Bud-newdiv}
A.~Bud.
\newblock Prym enumerative geometry and a {H}urwitz divisor in
  {$\overline{\mathcal{R}}_{2i}$}.
\newblock {\em Preprint, arXiv:2201.12009}, 2022.

\bibitem[Bud24]{BudKodPrym}
A.~Bud.
\newblock The birational geometry of {$\overline{\mathcal{R}}_{g,2}$} and
  {P}rym-canonical divisorial strata.
\newblock {\em Selecta Mathematica}, \textbf{30}, 2024.

\bibitem[Cat83]{Catanese}
F.~Catanese.
\newblock On the rationality of certain moduli spaces related to curves of
  genus $4$.
\newblock {\em Springer Lecture Notes in Mathematics}, \textbf{1008}:30--50,
  1983.

\bibitem[CLRW20]{tropicalPBN-4authors}
S.~Creech, Y.~Len, C.~Ritter, and D.~Wu.
\newblock {Prym–{B}rill–{N}oether {L}oci of {S}pecial {C}urves}.
\newblock {\em International Mathematics Research Notices},
  \textbf{2022}:2688--2728, 2020.

\bibitem[DCP95]{DeConciniPragacz}
C.~De~Concini and P.~Pragacz.
\newblock On the class of {B}rill-{N}oether loci for {P}rym varieties.
\newblock {\em Mathematische Annalen}, \textbf{302}:687--698, 1995.

\bibitem[Deb00]{DebarreLefschetz}
O.~Debarre.
\newblock Th\'eor\`emes de {L}efschetz pour les lieux de d\'eg\'en\'erescence.
\newblock {\em Bulletin de la Soci\'et\'e Math\'ematique de France},
  \textbf{128}:283--308, 2000.

\bibitem[DLC23]{Lelli-Chiesa-lowgenus}
S.~D'Evanghelista and M.~Lelli-Chiesa.
\newblock Double covers of curves on {N}ikulin surfaces.
\newblock {\em Preprint, arXiv:2305.06128, to appear in AMS Contemporary
  Mathematics}, 2023.

\bibitem[Dol85]{Dolgachev}
I.~Dolgachev.
\newblock Rationality of fields of invariants.
\newblock {\em Algebraic Geometry Bowdoin 1985, Proceedings of Symphosia in
  Pure Mathematics}, \textbf{46}-Part 2:3--16, 1985.

\bibitem[Dol96]{Dolgachev-Nikulin}
I.~Dolgachev.
\newblock Mirror symmetry for lattice polarized {K3} surfaces.
\newblock {\em Journal of Mathematical Sciences}, \textbf{81}:2599--2630, 1996.

\bibitem[Don84]{DonagiA5}
R.~Donagi.
\newblock The unirationality of $\mathcal{A}_5$.
\newblock {\em Annals of Mathematics}, \textbf{119}:269--307, 1984.

\bibitem[EH86]{limitlinearbasic}
D.~Eisenbud and J.~Harris.
\newblock Limit linear series: {B}asic theory.
\newblock {\em Inventiones mathematicae}, \textbf{85}:337--372, 1986.

\bibitem[EH87]{EisenbudHarrisg>23}
D.~Eisenbud and J.~Harris.
\newblock The {K}odaira dimension of the moduli space of curves of genus $\geq$
  23.
\newblock {\em Inventiones mathematicae}, \textbf{90}:359--387, 1987.

\bibitem[EH89]{IrreducibilityBN}
D.~Eisenbud and J.~Harris.
\newblock Irreducibility of some families of linear series with {Brill-Noether}
  number. {I}.
\newblock {\em Annales scientifiques de l'\'Ecole Normale Sup\'erieure},
  \textbf{22}:33--53, 1989.

\bibitem[Far00]{FarkasThesis}
G.~Farkas.
\newblock The birational geometry of the moduli space of curves.
\newblock {\em Academisch Proefschrift, Universitet van Amsterdam}, 2000.

\bibitem[Far09]{Farkosz}
G.~Farkas.
\newblock Koszul divisors on moduli spaces of curves.
\newblock {\em American Journal of Mathematics}, \textbf{131}:819--867, 2009.

\bibitem[FGSMV14]{FarkasGrushevsky}
G.~Farkas, S.~Grushevsky, R.~Salvati~Manni, and A.~Verra.
\newblock {Singularities of theta divisors and the geometry of $A_5$}.
\newblock {\em Journal of the European Mathematical Society},
  \textbf{16}:1817--1848, 2014.

\bibitem[FJP20]{FarPayneJensen}
G.~Farkas, D.~Jensen, and S.~Payne.
\newblock The {K}odaira dimensions of {$\overline{\mathcal{M}}_{22}$} and
  {$\overline{\mathcal{M}}_{23}$}.
\newblock {\em Preprint, arXiv:2005.00622}, 2020.

\bibitem[FJP24]{FarR13}
G.~Farkas, D.~Jensen, and S.~Payne.
\newblock The non-abelian {B}rill-{N}oether divisor on
  {$\overline{\mathcal{M}}_{13}$} and the {K}odaira dimension of
  {$\overline{\mathcal{R}}_{13}$}.
\newblock {\em Geometry {\normalfont \&} Topology}, \textbf{28}:803--866, 2024.

\bibitem[FL81]{FultonLazarsfConnected}
W.~Fulton and R.~Lazarsfeld.
\newblock {On the connectedness of degeneracy loci and special divisors}.
\newblock {\em Acta Mathematica}, \textbf{146}:271 -- 283, 1981.

\bibitem[FL10]{FarLud}
G.~Farkas and K.~Ludwig.
\newblock The {K}odaira dimension of the moduli space of {P}rym varieties.
\newblock {\em Journal of the European Mathematical Society},
  \textbf{12}:755--795, 2010.

\bibitem[FT16]{FarkasCastel}
G.~Farkas and N.~Tarasca.
\newblock Pointed {C}astelnuovo numbers.
\newblock {\em Mathematical Research Letters}, \textbf{23}:389--404, 2016.

\bibitem[Ful92]{FultonPragaczformula}
W.~Fulton.
\newblock {Flags, Schubert polynomials, degeneracy loci, and determinantal
  formulas}.
\newblock {\em Duke Mathematical Journal}, \textbf{65}:381 -- 420, 1992.

\bibitem[FV12]{FarVerNikulintheta}
G.~Farkas and A.~Verra.
\newblock Moduli of theta-characteristics via nikulin surfaces.
\newblock {\em Mathematische Annalen}, \textbf{354}:465--496, 2012.

\bibitem[FV16]{FarVerNikulin}
G.~Farkas and A.~Verra.
\newblock Prym varieties and moduli of polarized {N}ikulin surfaces.
\newblock {\em Advances in Mathematics}, \textbf{290}:314--328, 2016.

\bibitem[Gro61]{EGA2}
A.~Grothendieck.
\newblock \'el\'ements de g\'eom\'etrie alg\'ebrique : {II.} {\'etude} globale
  \'el\'ementaire de quelques classes de morphismes.
\newblock {\em Publications Math\'ematiques de l'IH\'ES}, \textbf{8}:5--222,
  1961.

\bibitem[Har84]{KodevenHarris1984}
J.~Harris.
\newblock On the {K}odaira dimension of the moduli space of curves, {II}. {T}he
  even-genus case.
\newblock {\em Inventiones mathematicae}, \textbf{75}:437--466, 1984.

\bibitem[HM82]{KodMg}
J.~Harris and D.~Mumford.
\newblock On the {K}odaira dimension of the moduli space of curves.
\newblock {\em Inventiones Mathematicae}, \textbf{67}:23--88, 1982.

\bibitem[IGS08]{Izadi}
E.~Izadi, M.~Lo Giudice, and G.~Sankaran.
\newblock The moduli space of \'etale double covers of genus 5 curves is
  unirational.
\newblock {\em Pacific Journal of Mathematics}, \textbf{239}:39--52, 2008.

\bibitem[JP21]{JensenPaynesurvey}
D.~Jensen and S.~Payne.
\newblock Recent {D}evelopments in {B}rill-{N}oether {T}heory.
\newblock {\em Preprint, arXiv:2111.00351}, 2021.

\bibitem[LCKV23]{Lelli-Chiesa-uniruled}
M.~Lelli-Chiesa, A.~L. Knutsen, and A.~Verra.
\newblock ({U}ni)rational parametrizations of $\mathcal{R}_{g,2}$,
  $\mathcal{R}_{g,4}$ and $\mathcal{R}_{g,6}$ in low genera.
\newblock {\em Preprint, arXiv:2310.16635}, 2023.

\bibitem[LU21]{tropicalPBN-Len_Ulirsch}
Y.~Len and M.~Ulirsch.
\newblock {Skeletons of {P}rym varieties and {B}rill–{N}oether theory}.
\newblock {\em Algebra {$\&$} Number Theory}, \textbf{15}:785--820, 2021.

\bibitem[MM83]{MoriMukai}
S.~Mori and S.~Mukai.
\newblock The uniruledness of the moduli space of curves of genus $11$.
\newblock {\em Springer Lecture Notes in Mathematics}, \textbf{1016}:334--353,
  1983.

\bibitem[Mum74]{MumfordPrym}
D.~Mumford.
\newblock Prym varieties {I}.
\newblock {\em Contributions to analysis}, pages 325--350, 1974.

\bibitem[Oss16]{Ossermandim}
B.~Osserman.
\newblock Dimension counts for limit linear series on curves not of compact
  type.
\newblock {\em Mathematische Zeitschrift}, \textbf{284}:69--93, 2016.

\bibitem[Oss19]{Ossermancompacttype}
B.~Osserman.
\newblock Limit linear series for curves not of compact type.
\newblock {\em Journal für die reine und angewandte Mathematik (Crelles
  Journal)}, \textbf{2019}:57--88, 2019.

\bibitem[P{\'e}r21]{Carlos}
C.~M. P{\'e}rez.
\newblock Prym curves with a vanishing theta null.
\newblock {\em Preprint, arXiv:2102.03435}, 2021.

\bibitem[Sch17]{SchwarzPrym}
Irene Schwarz.
\newblock Brill–{N}oether theory for cyclic covers.
\newblock {\em Journal of Pure and Applied Algebra}, \textbf{221}:2420--2430,
  2017.

\bibitem[Ste98]{Steffen}
F.~Steffen.
\newblock A generalized principal ideal theorem with an application to
  {B}rill-{N}oether theory.
\newblock {\em Inventiones mathematicae}, \textbf{132}:73--89, 1998.

\bibitem[Ver84]{VerraA5}
A.~Verra.
\newblock A short proof of the unirationality of $\mathcal{A}_5$.
\newblock {\em Indagationes Mathematicae (Proceedings)}, \textbf{87}:339--355,
  1984.

\bibitem[Ver08]{VerraA4}
A.~Verra.
\newblock On the universal principally polarized abelian variety of dimension
  $4$.
\newblock {\em In: Curves and abelian varieties (Athens, Georgia, 2007)
  Contemporary Mathematics}, \textbf{345}:253--274, 2008.

\bibitem[vGS07]{Nikulin}
B.~van Geemen and A.~Sarti.
\newblock Nikulin involutions on k3 surfaces.
\newblock {\em Mathematische Zeitschrift}, \textbf{255}:731 -- 753, 2007.

\bibitem[Wel85]{Welters}
G.~Welters.
\newblock A theorem of {G}ieseker-{P}etri type for {P}rym varieties.
\newblock {\em Annales Scientifiques de l'École Normale Supérieure},
  \textbf{18}:671--683, 1985.

\end{thebibliography}
\bibliographystyle{alpha}
\Addresses
\end{document}